\documentclass[12pt,a4paper,reqno]{amsart}

\usepackage[T1]{fontenc}
\usepackage[utf8]{inputenc}
\usepackage{amsthm, textcomp, amsmath, amsfonts,  amssymb }
\usepackage[pdftex]{graphicx}
\usepackage{enumerate}
\usepackage{url}
\usepackage{comment}
\usepackage[colorlinks]{hyperref}
\usepackage[a4paper,left=2cm,right=2cm]{geometry}
\allowdisplaybreaks[4]
\usepackage{ulem}
\usepackage{xcolor}

\usepackage{caption}
\usepackage{subcaption}
\usepackage{orcidlink}

\author[G. S. Alberti]{G. S. Alberti$^1$ \orcidlink{0000-0002-8612-3663}}
\author[E. De Vito]{E. De Vito$^1$ \orcidlink{0000-0002-4320-3292}}
\author[B. Gariboldi]{B. Gariboldi$^{2}$ \orcidlink{0000-0001-8714-4135}}
\author[G. Gigante]{G. Gigante$^{2}$ \orcidlink{0000-0002-1642-679X}}
\address{\phantom{1}$^1$MaLGa Center, Department of Mathematics, University of Genoa, Via Dodecaneso 35, 16146, Genova, Italy}
\email{giovanni.alberti@unige.it}
\email{ernesto.devito@unige.it}
\address{\phantom{i}$^2$Dipartimento di Ingegneria Gestionale, dell'Informazione e
della Produzione, Universit{\`a} degli Studi di Bergamo, Viale G. Marconi 5, 24044,
Dalmine BG, Italy}
\email{biancamaria.gariboldi@unibg.it}
\email{giacomo.gigante@unibg.it}

\begin{document}

\newtheorem{Theorem}{Theorem}
\newtheorem{Corollary}[Theorem]{Corollary}
\newtheorem{Proposition}[Theorem]{Proposition}
\newtheorem{lemma}[Theorem]{Lemma}
\newtheorem{prop}[Theorem]{Proposition}
\theoremstyle{remark}
\newtheorem{remark}[Theorem]{Remark}
\theoremstyle{definition}
\newtheorem{ex}[Theorem]{Example}
\newtheorem{definition}[Theorem]{Definition}
\newtheorem{assumption}[Theorem]{Assumption}

\newcommand{\R}{\mathbb{R}}

\title[Sampling theorems for inverse problems on Riemannian manifolds]{Sampling theorems for inverse problems \\on Riemannian manifolds}



\date{January 28, 2026}

\begin{abstract}
We consider inverse problems  consisting of the reconstruction of an unknown signal $f$ from noisy measurements $y=Ff+\text{noise}$, where $Ff$ is a function on a Riemannian manifold without boundary $\mathcal M$. We consider the case when only pointwise samples are available, namely $y_j = (Ff)(x_j)+\eta_j$, where $\{x_j\}_{j=1}^n\subseteq\mathcal M$ is a Marcinkiewicz-Zygmund family. We derive sampling theorems providing explicit bounds on the reconstruction error depending on $n$, the smoothness of $f$ and the properties of $F$. We study in detail the case when $F$ is a convolution on a compact two-point homogeneous space. As a corollary, we state a sampling theorem for convolutions on the two-dimensional sphere, and discuss four relevant examples related to terrestrial and celestial measurements.

\end{abstract}

\maketitle


\section{Introduction}

\subsection{Sampling and approximation}\label{sub:Sampling-intro}

The classical Shannon sampling theorem, a cornerstone of signal processing and harmonic analysis, asserts that a bandlimited function $f \in L^2(\mathbb{R})$ with spectral support contained in an interval $[-\Omega, \Omega]$ is entirely determined by its values on the discrete set $\{x_j=j/(2\Omega)\}_{j\in\mathbb{Z}}$, and can be exactly reconstructed via the Shannon-Whittaker interpolation formula \cite{Whittaker1915,Nyquist1928,Shannon1949}. This fundamental result highlights the intrinsic link between the bandwidth $\Omega$ — the measure of frequency content — and the density $\frac{1}{2\Omega}$ of the samples required for perfect reconstruction. 
In higher dimensions, analogous results hold for functions in $L^2(\mathbb{R}^d)$ whose Fourier transforms are supported in compact sets: such functions can be reconstructed from their values on suitable lattices, provided the sampling rate exceeds the Nyquist threshold determined by the volume of the spectral support \cite{Landau1967,BenedettoFerreira2001}.

On compact domains such as the torus $\mathbb{T}^d$, the sphere $S^d$, or a compact Riemannian manifold  $\mathcal{M}$ without boundary, where the frequency domain is discrete, the situation becomes even more structured: bandlimited functions $f$ become finite-dimensional, and exact recovery from a finite number of measurements $\{f(x_j)\}_j$ is possible, provided the samples $\{x_j\}_j$ are suitably taken \cite{mcewen-wiaux-2011,Gro,LW,LLG}. While the case of the torus is trivial because it is enough to consider uniform Cartesian grids \cite{Gro}, the construction of the samples $\{x_j\}_j$ is, in general, a difficult problem. This can be formalized with the concept of Marcinkiewicz-Zygmund family, whose construction is guaranteed whenever the sampling space can be divided into regions, each with a prefixed volume and small diameter \cite{MNW,maggioni,FM,FM1}. This type of partition, on the other hand, exists in every connected compact orientable Riemannian manifold without boundary \cite{GL, cubature}, and is explicit in simple settings, such as for the two-dimensional sphere \cite{RSZ}. As with the Shannon theorem, the bandwidth of $f$ is explicitly linked with the density of the sampling points $\{x_j\}_j$. All these results have an ``approximation'' counterpart in the cases where $f$ is not bandlimited but smooth, an assumption that is quantified by the decay of the generalized Fourier coefficients of $f$.

\subsection{Inverse problems}

Inverse problems appear when an unknown signal $f$ belonging to a space $X$ has to be recovered from indirect measurements of the form $Ff$, where $F\colon X\to Y$ is the so-called forward map, and $Y$ is the measurement space \cite{bertero2021introduction,isakov-2017,handbook}. Inverse problems are typically ill-posed: whenever measurements are noisy, so that only an approximation $y$ of $Ff$ is available, where $\|y-Ff\|_Y$ is small, the reconstruction of an approximation of $f$ is problematic, because the inverse of $F$ is not continuous. In the linear setting, this is often a consequence of the compactness of $F$. For example, this is the case whenever $F$ is given as an integral operator with a suitable kernel, such as a convolution operator or the Radon transform \cite{natterer-2001,radon100}. The classical method to address ill-posedness is regularization \cite{engl-etal-1996,kaltenbacher-neubauer-scherzer-2008}.

Even if $X$ and $Y$ are usually infinite-dimensional function spaces, which is due to the fact that the models often involve integral or differential operators, the available measurements are always finite dimensional \cite{kaipio-somersalo-2007,lassas-etal-2009,adcock-hansen-etal-2013,kekkonen-lassas-siltanen-2014,Giovanni,alberti-arroyo-santacesaria}. 
In many settings, assuming that $Y$ is a space of functions defined on a  manifold $\mathcal{M}$, the finite measurements can be written as noisy pointwise samples of $Ff$ on a suitable family of points $\{x_j\}_j\subset\mathcal{M}$, namely,
\begin{equation}\label{eq:intro-yj}
y_j = (Ff)(x_j)+\eta_j,\qquad j=1,\dots,N,
\end{equation}
where $\{\eta_j\}_j$ represents measurement noise. This problem is often studied in a random design setting, in which the points $x_j$ are sampled i.i.d.\ with respect to a fixed probability distribution on $\mathcal M$ \cite{dashti-harris-stuart-2012}. Several results exist in this context, regarding approaches based on statistical inverse learning (mostly for linear problems)  \cite{blanchard-mucke-2018,abhishake-blanchard-mathe-2020,bubba-ratti-2022,bubba-burger-helin-ratti-2023,helin-2024,AbhishakeHelinMücke+2025+201+244}
and approaches studying the statistical consistency for nonlinear problems \cite{abraham-nickl-2019,giordano-nickl-2020,nickl-2023,nickl-paternain-2023}. Most of these works are based on smoothness/Gaussian priors on $f$; for sparsity-based priors, several works have investigated the use of compressed sensing techniques to further reduce the number of measurements in inverse problems, see e.g.\ \cite{herrholz-teschke-2010,alberti-santacesaria-2021,ebner-haltmeier-2023,lazzaro-morigi-ratti-2024,alberti2025compressed,alberti2025compressed2}. In the context of a deterministic design, we refer the reader to \cite{stefanov-2020,stefanov-tindel-2023,monard-stefanov-2023}, where the authors study the case of Fourier Integral Operators $F$.

\subsection{Our contribution: sampling for inverse problems}

In this work, we address the inverse problem consisting of the recovery of $f$ from the measurements $\{y_j\}_j$ given in \eqref{eq:intro-yj} by applying the techniques based on sampling and approximation discussed in $\S$\ref{sub:Sampling-intro}. In contrast to the approaches based on random design points mentioned above, here the points $\{x_j\}_j$ are deterministic and chosen to form a Marcinkiewicz-Zygmund family, which as we mentioned before can be achieved e.g. in any orientable compact connected Riemannian manifold without boundary. We consider abstract inverse problems  under suitable smoothness assumptions on the forward map $F$. We prove recovery estimates and explicitly quantify the error as a function of the noise level and of the number of samples.

Furthermore, we leverage the flexibility of considering general Riemannian manifolds beyond the Euclidean setting and consider convolution operators $F$ on compact two-point homogeneous spaces, which allows us to write $F$ as Fourier multiplier  (with respect to the basis given by the eigenfunctions of the Laplace-Beltrami operator on $\mathcal M$). In this setting, the inverse problem consists in the reconstruction of a function $f$ from noisy pointwise samples of $(h*f)(x_j)$, where $h$ is a given filter.
As a key application, we consider the convolution on the two-dimensional sphere \cite{vanrooij-ruymgaart-1991,healy-etal-1998,kim-koo-2002,kerkyacharian-etal-2011,vareschi-2014}, in the spirit of the results obtained in \cite{hielscher-quellmalz-2015} with pointwise samples. We obtain explicit error estimates, and quantify how the smoothness of $f$ and the degree of ill-posedness of $F$ (namely, the decay of the generalized Fourier coefficients of $h$) affect the reconstruction error and the number of required samples. These estimates are then applied to four examples of deconvolution problems on the sphere  \cite{tesi,articolotesi}.

The focus of our investigation on two-point homogeneous spaces is justified by the fact that, on the one hand, they contain  the sphere $S^d$ as a particular case (for which the construction of design points is anyway nontrivial), and  all the construction can be carried out in the abstract setting with little additional effort. Indeed, thanks to Wang's theorem  \cite{Wang}, these spaces are all and only the compact symmetric spaces of rank one. On the other hand, going beyond two-point homogeneous spaces is not trivial, because convolutions are harder to represent, and in general the harmonic analysis for spaces of higher rank is more difficult.

\subsection{Structure of the paper}
In Section~\ref{sec:sphere} we anticipate our results specialized for the convolution problem on the sphere, in order to better illustrate the abstract approach presented in the rest of the paper. These results are applied to four different examples in Section~\ref{sec:examples}.  In Section~\ref{sec:sampling} we review some basic facts about sampling and approximation for functions on manifolds, based on Marcinkiewicz-Zygmund families. In Section~\ref{sec:sampling_IP} we apply these results to abstract inverse problems. Finally, in Section~\ref{sec:convolution} we apply the abstract derivation to convolutions on compact two-point homogeneous spaces.

\section{Convolution on the sphere}\label{sec:sphere}

In this section, motivated by several applications, we anticipate our results specialized for the convolution operator on the sphere of the Euclidean space. 

\subsection{The abstract result for the sphere \texorpdfstring{$S^d$}{Sd}}

We briefly introduce the notation and setting. 
Let $S^d$ be the sphere in $\mathbb R^{d+1}$ endowed with the
Riemann measure $\mu$ normalized as $\mu(S^d)=1$. 
The orthogonal group $SO(d+1)$ acts transitively on $S^d$ and its
action is denoted by $gx$ where $g\in SO(d+1)$ and $x\in S^d$.
Letting $o\in  S^d$ denote the north pole, the stability subgroup at $o$ is 
\[\{g\in SO(d+1) : go=o\}= SO(d), \]
so that we identify $S^d$ with the quotient space $SO(d+1)/SO(d)$
by means of the map
\[
SO(d+1) \ni g\mapsto g o \in  S^d.
\]
For a fixed  filter function $h\in L^1(S^d,\mu)$, the convolution operator
$F\colon L^2(S^d)\to L^2(S^d)$ by $h$, $Ff = h*f$, is defined as 
\[
(F f)(x) = \int_{SO(d+1)} h(g^{-1}x) f(g o) dg \,,
\]
where $dg$ is the Haar measure of $SO(d+1)$ normalized to $1$.
It is well known that the convolution operator has a nice structure in terms of spherical
harmonics. Indeed, let $\Delta$ be the  Laplace-Beltrami operator on $S^d$,
regarded as a self-adjoint  positive operator on $L^2(S^d)$.  Its
(distinct) eigenvalues are
\[
\lambda_m^2 = m(m+d-1) \qquad m\in \mathbb N. 
\]
For each $m\in \mathbb N$,  let $\mathcal{H}_{m}$ be the space of the 
eigenfunctions  of $\Delta$ with eigenvalue  $\lambda_m^2$. The space
$ \mathcal{H}_{m}$ is finite dimensional, and we denote its dimension by
\[
\dim\mathcal{H}_{m}=\delta_m= \binom{d+m}{d}- \binom{d+m-2}{d}\asymp m^{d-1}
\] 
and the corresponding orthogonal projection by $\Pi_m$.
Finally, let
\begin{equation}\label{diffusion_pol}
    \mathcal Q_m  = \bigoplus_{\ell=0}^m \mathcal H_\ell \qquad m\in\mathbb N,
\end{equation}
be the space of diffusion polynomials of bandwith
$\lambda_m$. 

Assume now that the filter function  $h\in L^2(S^d)\subset L^1(S^d)$. By setting
  \begin{equation} \label{eq:1}
b_m = \frac{1}{\delta_m} (\Pi_m h)(o) \qquad m\in \mathbb
N,
\end{equation}
the convolution operator $F$ with filter $h$ reads as
\begin{equation}\label{eq:2}
 F f = \sum_{m=0}^\infty b_m \Pi_m f \qquad f\in L^2(S^d),
\end{equation}
{\it i.e.} $F$ is a Fourier multiplier.  Note that, for any bounded sequence
$\{b_m\}_{m\in\mathbb N}$, Eq.~\eqref{eq:2} defines a bounded operator,
which in general is not a convolution operator, compare with the definition of generalized convolution operators in \cite{quellmalz2018cone}. For example,  if $b_m=1$  for all $m\in \mathbb N$, then $F$ is the identity
operator. We stress that  our results hold true also for this larger class of operators.

We need to recall the definition of Sobolev
spaces $H^\sigma(S^d)$ on the sphere. For $\sigma\geq 0$, we have that
\[
H^\sigma(S^d)=\left\{f\in L^2(S^d):\sum_{m=0}^{+\infty}
  (1+{\lambda}_m^2)^\sigma \lVert\Pi_m f   \rVert^2<+\infty\right\} 
\]
is a Hilbert space with respect to the norm 
\[
\|f\|^2_{H^\sigma(S^d)}=\sum_{m=0}^{+\infty}(1+{\lambda}_m^2)^\sigma\lVert \Pi_m f
  \rVert^2 .
\]

The following proposition provides some smoothing properties of the convolution operator; it follows from the general case treated in Proposition~\ref{operator}. 

\begin{Proposition}\label{operator-1}
Choose $h\in L^2(S^d)$, set $b_m$ as in~\eqref{eq:1} and assume that
  \begin{equation}
|b_m|\le c(1+ {\lambda}^2_m )^{-\frac{\gamma}2}\label{eq:4}
\end{equation}
for some $c,\gamma\ge 0$. Fix $\omega\geq 0$ and $\sigma=\omega+\gamma$, then
$F$ is a bounded operator from $H^\omega(S^d)$ into $ H^\sigma(S^d)$,
with norm bounded by $c$. 
\end{Proposition}

The next result shows the existence of a nice discretization family of
points for every diffusion polynomial $p\in\mathcal Q_m $, see Eq.~\eqref{diffusion_pol}. It is a particular case of Theorems~\ref{MZ polys} and \ref{partition}.
\begin{Proposition}\label{discr}
Fix $0<\varepsilon<1$ and $m\in\mathbb N$. There exists a
finite family $\{ x_j\}_{j=1}^{L_m}$ of distinct points of $S^d$ such that
  \begin{equation}
(1-\varepsilon) \|p\|_2^2\le \frac{1}{L_m} \sum_{j=1}^{L_m}|p(x_{j})|^2 \le
  (1+\varepsilon) \|p\|_2^2, \qquad p \in \mathcal Q_m\label{eq:8}
\end{equation}
and  $L_m= \lceil C_1 ({\lambda}_m^2)^\frac d2 \rceil\asymp m^d$ for
  some constant $C_1>0$. 

\end{Proposition}
  
It is possible to provide an explicit construction of the
discretization family $\{ x_j\}_{j=1}^{L_m}$; we describe the case $d=2$ in
Section~\ref{sec:an-example-family}. 
The points $\{ x_j\}_{j=1}^{L_m}$ are called a Marcinkiewicz-Zygmund
family. The map
\[
p \in \mathcal Q_m \longmapsto \{p(x_j)\}_{j=1}^{L_m} \in \mathbb{C}^{L_m}
\]
is $\varepsilon$-close to an isometry, which allows for the stable recovery of $p \in \mathcal Q_m$ from its pointwise samples  $\{p(x_j)\}_{j}$.

We now state our main result for the sphere, which is a particular case of Theorem~\ref{finale} and Corollary \ref{inverso}.  The focus is the reconstruction of $f^\dagger$, where the measurement data are the noisy values of the function $Ff^\dagger=h*f^\dagger$ sampled at the discretization points  $\{x_j\}_{j=1}^{L_m}$, namely,
\[
y_j = (h* f^\dagger)(x_j) + \eta_j,
\]
where the noise $\eta_j$ is bounded by $\beta$.
\begin{Theorem}\label{main}
Let $h\in L^2(S^d)$  such that $b_0\neq 0$ and Eq.~\eqref{eq:4}
holds true for some $\gamma\ge0$.  Fix $f^\dagger\in  H^\omega(S^d)$  with $\omega\ge0$, and assume that for some $\zeta\ge0$,
$\sigma=\omega+\gamma>d/2+\zeta$.  Given $m\in\mathbb N$, let $\{
x_j\}_{j=1}^{L_m}$ be the discretization  family of points given by
Proposition~\ref{discr} and let $\{y_j\}_{j=1}^{L_m}$ be the corresponding
family of noisy outputs such that
\[
 |y_j- Ff^\dagger(x_j)|\leq\beta \quad j=1,\ldots,L_m
\]
where $\beta \ge 0$ is the noise level.  Choose $p_m^\beta\in \mathcal Q_m$ as 
  \begin{equation}
p_m^\beta \in \operatornamewithlimits{argmin}_{p\in\mathcal Q_m} 
\frac{1}{L_m} \sum_{j=1}^{L_m}|y_j-Fp(x_j)|^2  \label{eq:3},
\end{equation}
then 
  \begin{align}
\Vert Ff^\dagger - Fp_m^\beta\rVert_{H^\zeta(S^d)} \leq &\sqrt{\frac{1
  +\kappa}{\sigma-\zeta-\frac d2}}\, \Vert Ff^\dagger
\rVert_{H^\sigma(S^d)}  (1+\lambda_m^2)^{-\frac{\sigma-\zeta}2+\frac d4} +
\sqrt{\kappa} \beta(1+\lambda_m^2)^{\frac{\zeta}2}\label{eq:5}
\end{align}
where $\kappa= \frac{1+\varepsilon}{1-\varepsilon}$. If furthermore $\zeta\ge\gamma$ and for some constant $c_0>0$,
\begin{equation}\label{below}
|b_m|\ge c_0(1+\lambda_m^2)^{-\zeta/2} \text{ for all }m\ge 0,
\end{equation}
then we have
 \begin{equation}\label{eq:29}
    \|f^\dagger-p_m^\beta\|_2\le c_0^{-1}\|Ff^\dagger-Fp_m^\beta\|_{H^\zeta(S^d)}.
\end{equation}
\end{Theorem}

The function $p_m^\beta$ is a solution of a least squares problem~\eqref{eq:3} having the space of diffusion polynomials (with a given bandwith) $\mathcal Q_m$ as hypothesis space (as in regularization by projection \cite{kaltenbacher-neubauer-scherzer-2008}). Given the hypothesis space $\mathcal Q_m$, the minimizer $p_m^\beta$ of \eqref{eq:3}  depends on the measured noisy data $y_j= Ff^\dagger(x_j)+\eta_j$, where the noise $\eta_j$ is bounded by $\beta$. Note that
$\mathcal Q_m$ has dimension $\sum_{\ell=0}^m\delta_\ell\asymp m^{d}$ and the number of
discretization points is $L_m\asymp m^d$. Our result shows that,  for a sufficiently small noise level $\beta$, $Fp_m^\beta$
is a good approximation (in the $H^\zeta$-norm) of the true value $h*f^\dagger$, and that, under hypothesis \eqref{below}
on $h$, $p_m^\beta$ is a good approximation (in the $L^2$ norm) of $f^\dagger$.
Eqs.~\eqref{eq:5} and~\eqref{eq:29} provide a bound on the reconstruction error. 
The parameter $\varepsilon$ does not play a crucial role in the estimates, and we can choose $\varepsilon=\frac12$, so that 
 $\kappa=3$. On the other hand, $\sigma$ is related to the apriori information on the smoothness of the true solution $f^\dagger$, while $\zeta$ is related to the norm with which the error is measured. The parameter to be tuned is $m$ as a function of the noise level $\beta$. 
In order to highlight this dependence, let us write the recovery estimate without the explicit constants, by using that $\lambda_m^2\asymp m^2$ for $m\geq 1$: 
\[
\|f^\dagger-p_m^\beta\|_2 \lesssim m^{\zeta-\sigma+\frac{d}{2}} + \beta m^\zeta.
\]
For a fixed noise level $\beta$, this leads us to choose $m$ in such a way that the two terms have the same order, namely $m^{\zeta-\sigma+\frac{d}{2}}\asymp \beta m^\zeta$, so that 
\begin{equation}\label{eq:mandbeta}
    m= \left\lceil\beta^{-\frac{1}{\omega+\gamma-\frac{d}{2}}}\right\rceil,
\end{equation}
where we used that $\sigma=\omega+\gamma$. Plugging this into the bound above yields
\begin{equation}\label{eq:asymptotic}
\|f^\dagger-p_m^\beta\|_2 \lesssim \beta^{1-\frac{\zeta}{\omega+\gamma-\frac{d}{2}}}.
\end{equation}
As expected, since $\omega+\gamma-\frac{d}{2}>\zeta$ by assumption, as $\beta\to 0$, we have $p_m^\beta\to f^\dagger$, provided that the number of measurements increases accordingly to \eqref{eq:mandbeta}. Furthermore, it is possible to understand the role of the parameters involved in these estimates:
\begin{itemize}
    \item the parameter $\omega$ such that $f^\dagger\in H^\omega(S^2)$ quantifies the smoothness assumption on $f^\dagger$;
    \item the parameters $\gamma$ and $\zeta$ quantify the smoothness of $h$ through the decay of its generalized Fourier coefficients in \eqref{eq:4} and  \eqref{below}, and will usually be of the same order.
\end{itemize}
Therefore, according to \eqref{eq:mandbeta}, fewer samples are needed as $\gamma$ and $\omega$ increase (in other words, as the problem becomes more ill-posed and as $f^\dagger$ becomes smoother). Similarly, in light of \eqref{eq:asymptotic}, a smoother $f^\dagger$ (larger $\omega$) yields a smaller reconstruction error, but a more ill-posed problem (larger $\gamma$ and $\zeta$) yields a larger error.

In \cite{hielscher-quellmalz-2015}, results similar to ours were obtained for the two-dimensional sphere, but in a setting with statistical noise and a different choice of approximating polynomial. In particular, exact knowledge of the multiplier associated with the operator is required, and the sampling points must form a spherical design—a construction that is more challenging than that of the Marcinkiewicz-Zygmund
family provided by Proposition~\ref{discr}.

\subsection{The particular case of the two-dimensional sphere \texorpdfstring{$S^2$}{S2}}\label{sec:an-example-family}

We provide an example of a discretization family for the two-dimensional sphere. It can be shown (see \cite[Theorem 5.1]{FM}) that  if one can divide the manifold (in this case the sphere $S^2$) into $N$ regions with the same measure and bounded diameter, then every choice of points, one in each region, gives a discretization family, as in Proposition~\ref{discr}. For the $2$-dimensional sphere, we describe an explicit partition  \cite{RSZ}. Observe that this type of costruction has been done also for the case of the $d$-dimensional sphere (see \cite{leopardi}).
 Fix $m\in\mathbb N$ and let $N=L_m$ be the size of the
discretization family.  Let 
\begin{itemize}
    \item $\theta_0=\arccos(1-50/N)$;
    \item $s$ be the largest odd integer smaller than or equal to $ \sqrt{\pi N}/2$;
    \item $\Delta\theta=(\pi-2\theta_0)/s$;
    \item $\theta_k'=\theta_0+k\Delta\theta$ for $k=1,2,\ldots,s$, $\theta'_{s+1}=\pi$;
    \item and $y_k=N(\cos\theta_{k-1}'-\cos\theta_k')/2$, for $k=1,2,\ldots,s$. 
\end{itemize}
It can be proved that there exists a symmetric sequence of integers $\{\ell_i\}_{i=1}^s$ such that
\begin{enumerate}
\item $\sum_{i=1}^s{\ell}_i=\sum_{i=1}^sy_i=N-50$.
\item $|y_1-{\ell}_1|=|y_s-{\ell}_s|\le1/2$ and $|y_i-{\ell}_i|\le1$ for $i=2,\ldots s-1$.
\item $\left|\sum_{i=1}^k(y_i-{\ell}_i)\right|\le1/2$, for $k=1,\ldots,s$. 
\end{enumerate}
Finally set ${\ell}_0={\ell}_{s+1}=25$.
It is shown in \cite{RSZ} that for $N\ge 50$, the regions of the sphere described in spherical coordinates as
$D^N_{k,j}=[\theta_{k-1},\theta_k]\times[2\pi(j-1)/{\ell_k},2\pi
j/{\ell}_k]$, for $k=0,\ldots,s+1$,
$j=1,\ldots,{\ell}_k$, where
$\theta_k=\arccos(1-(2/N)\sum_{i=0}^k{\ell}_i)$ for
$k=0,\ldots, s+1$, and $\theta_{-1}=0$, form a partition of the sphere
of regions with equal area $1/N$, each contained in a spherical cap of radius bounded above by $c_4N^{-1/2}$, for some $c_4>0$. It can be easily shown that all regions contain a spherical cap of radius $c_3N^{-1/2}$, for some $c_3>0$. 

Therefore we can give the following constructive version of Proposition~\ref{discr} (see also Theorem~\ref{MZ polys}).
\begin{Proposition}\label{discr1}
Fix $0<\varepsilon<1$ and $m\in\mathbb N$. For every possible choice of $x_{k,j}\in D_{k,j}^{L_m}$, the family $\{x_{k,j}\}_{k=0, j=1}^{s+1,\ell_k}$ is such that
  \begin{equation}
(1-\varepsilon) \|p\|_2^2\le \frac{1}{L_m} \sum_{k=0}^{s+1}\sum_{j=1}^{\ell_k}|p(x_{k,j})|^2 \le
  (1+\varepsilon) \|p\|_2^2, \qquad p \in \mathcal Q_m\label{eq:8bis}
\end{equation}
whenever  $L_m= \lceil C_1 m(m+1) \rceil$ for
  some constant $C_1>0$.

\end{Proposition}

\section{Four examples}\label{sec:examples}

In this section, we discuss four examples related to Theorem~\ref{main} for the two-dimensional sphere $S^2$, which are motivated by real-world applications. For further details, the reader is referred to \cite{tesi, articolotesi}. 

\subsection{Sea surface temperature anomalies.} This experiment
\cite{tesi, articolotesi} simulates the global map of temperature
anomalies of the sea surface produced by NASA's Aqua satellite in
January 2011, by sampling it in $N=6745$ locations. Since temperatures
typically have smooth variations, one may assume the anomaly function
$f^\dagger$ to be an element of $H^\omega(S^2)$ for some
$\omega>1$. The inversion problem here reduces to $F$ being the
identity operator, that is $b_m=1$ for every $m$
(as already noted, our
result holds true for any Fourier multiplier of the form \eqref{eq:2}).

 We may therefore set $\gamma=0$, and apply Theorem~\ref{main} with $\zeta=0$ to obtain
\begin{align*}
\|p_m^\beta-f^\dagger\|_2&\le \sqrt{1+\kappa}\|f^\dagger\|_{H^{\omega}(S^2)}\frac{1}{(\omega-1)^{\frac12}}\frac1{(1+m(m+1))^{\frac{\omega-1}2}}+\sqrt{\kappa}\beta.
\end{align*}

\subsection{ Wildfires and Deforestation.} This example deals with global density maps of trees and
wildfires across the globe for the year 2016. In both experiments, the data used in \cite{tesi, articolotesi} consists of tree and fire counts recorded by
NASA’s Aqua and Terra satellites. In \cite{tesi, articolotesi}, it is assumed that the spatial density $f^\dagger$ of trees (respectively fires) is a function in $H^\omega(S^2)$ with $\omega>1$, where we model the Earth surface by $S^2$. In both cases, the inverse problem consists in recovering $f^\dagger$ from a finite number of pointwise evaluations of $Ff^\dagger=h\ast f^\dagger$, where $h$ is the characteristic function of a small region around the North Pole. Here we propose to replace this choice of $h$ with the characteristic function of a spherical cap of angle $\vartheta_0\le\pi/2$, so that $h(x)=\chi_{[0,\vartheta_0]}(\rho(o,x))$ where $\rho(o,x)$ is the angular distance of $x$ from the North Pole. 
With this choice, by Eq.\ \eqref{bm} below and Rodrigues' formula (see \cite[(4.3.1)]{Sz}), if $m\geq 1$, 
\[b_m=\frac{1}{m}P^{1,1}_{m-1} (\cos \vartheta_0) \left(\sin \frac{\vartheta_0}{2}\right)^2\left(\cos \frac{\vartheta_0}{2}\right)^2.  \]

Let us now prove the upper bound \eqref{eq:4} with $\gamma=3/2$.
Kogbetliantz inequality \cite[eq. (1.3)]{forster} applied to our setting reads as
\[
\left|P_{m-1}^{1,1}(\cos\vartheta_0)(\sin\vartheta_0)^{\frac32}\right|
\le\frac{4\Gamma\left(m+\frac12\right)}{\Gamma\left(\frac 32\right)(m+1)\Gamma(m)}
\]
for all $\vartheta_0\le \pi/2$ and for all $m\ge1$. It follows that for $m\ge1$
\begin{align*}
|b_m|
&\le \left|\frac{(\sin\vartheta_0)^{\frac12}}{4m}P^{1,1}_{m-1} (\cos \vartheta_0) \left(\sin \vartheta_0\right)^{\frac 32}\right|\\
&\le \left|\frac{(\sin\vartheta_0)^{\frac12}}{m}\frac{\Gamma\left(m+\frac12\right)}{\Gamma\left(\frac 32\right)(m+1)\Gamma(m)}\right|\le 
\frac{3^{\frac34}}2\frac{(\sin\vartheta_0)^{\frac12}}{(1+m(m+1))^{\frac34}}.
\end{align*}
Since $b_0=(1-\cos\vartheta_0)/2\le (\sin\vartheta_0)^2/2\le 3^{3/4}(\sin\vartheta_0)^{1/2}/2$, we can deduce that indeed for all $m\ge0$,
\[
|b_m|\le\frac{3^{\frac34}}2\frac{(\sin\vartheta_0)^{\frac12}}{(1+m(m+1))^{\frac34}}.
\]

In Figure~\ref{fig:a} we illustrate this bound numerically.

We now move to studying the lower bound \eqref{below} with $\zeta=3/2$.
It can be proved (see \cite{bilyk} and also \cite{BGGM1}) that $x=\cos(2\pi/41)$ is a $3/2$-gegenbadly approximable number, which in our notation means that there exists a constant $c$ such that for every $m\geq 1$ 
\[|P^{1,1}_{m-1}(x)|\geq cm^{-\frac12}.\]
In other words, if we set $\vartheta_0=2\pi/41$, then for every $m\geq 1$, 
\[|b_m|\geq c_0\frac{1}{(1+m(m+1))^\frac34},\]
see also \cite{MR1992465, MR57963}. In particular, in \cite[section 6]{MR1992465},  Rubin proves the injectivity of $F$ from $H^\omega(S^2)$ to $H^{\omega+\frac32}(S^2)$ for $\vartheta=\pi/4$ and $\vartheta=\pi/3$, but  no bounds on the norm of the inverse are derived.
Numerical simulations (Figure~\ref{fig:b}) give an estimate $c_0\ge 0.4*10^{-3}$. (Similar remarks hold for other values of $\vartheta_0$).

\begin{figure}
     \centering
     \begin{subfigure}[b]{0.46\textwidth}
         \centering
         \includegraphics[width=\textwidth]{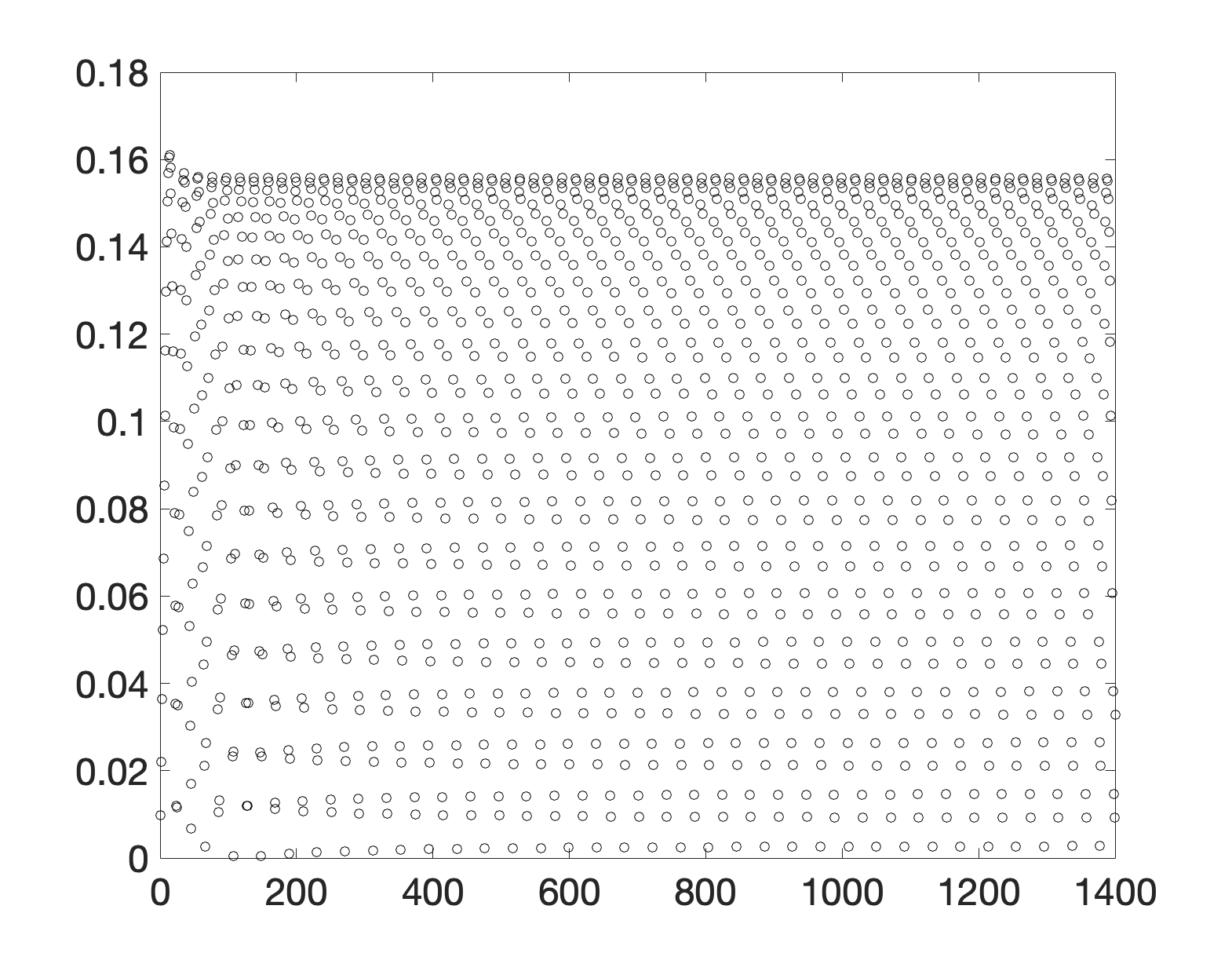}
         \caption{All values}
         \label{fig:a}
     \end{subfigure}
     \hspace{.5cm}
     \begin{subfigure}[b]{0.46\textwidth}
         \centering
         \includegraphics[width=\textwidth]{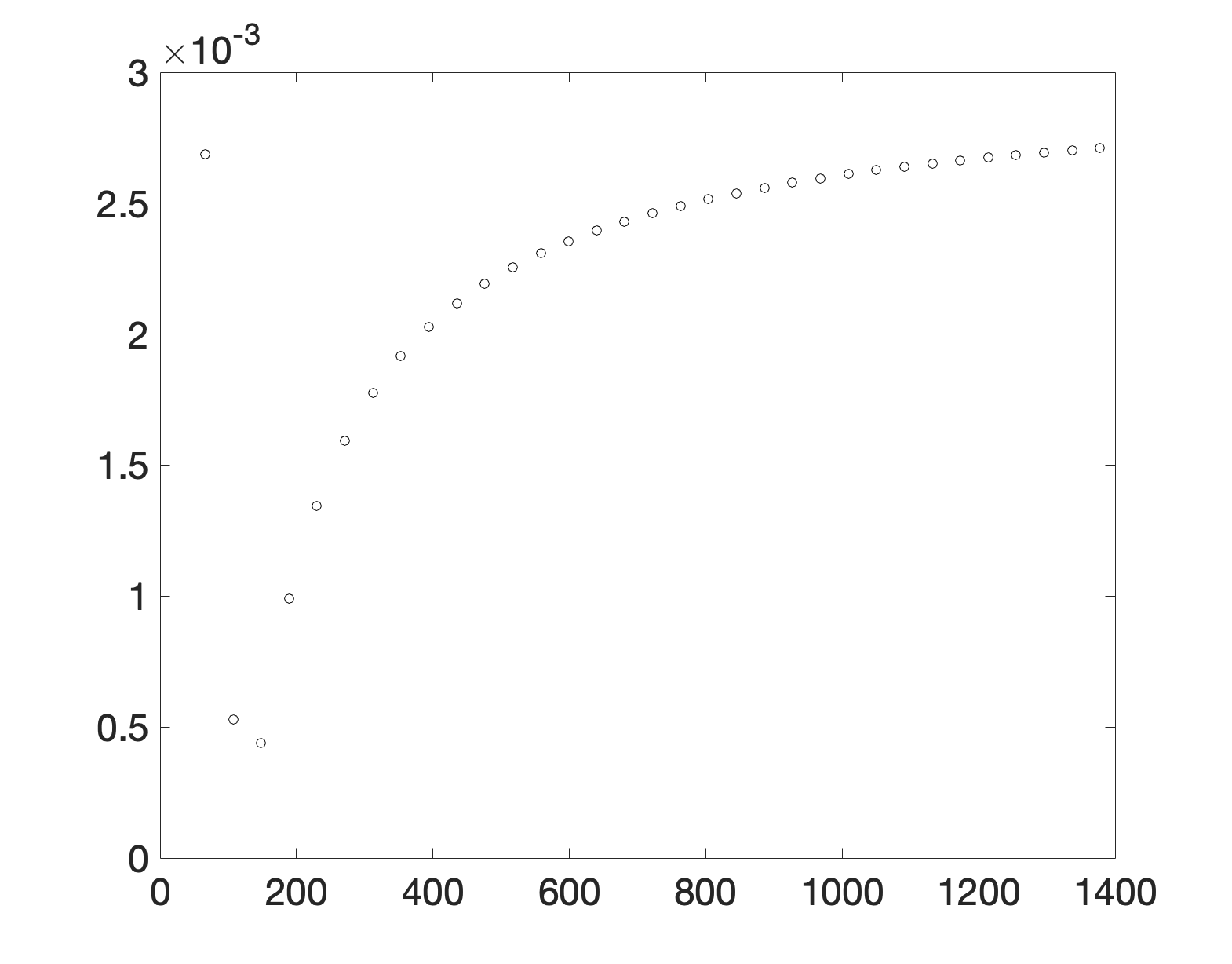}
         \caption{Zoom-in on the smallest values}
         \label{fig:b}
     \end{subfigure}
        \caption{Plot of the sequence $(1+m(m+1))^{3/4}|b_m|$ for $\vartheta_0=2\pi/41$ and  $m=1,\dots,1400$.}
        \label{fig:wildfires}
\end{figure}

Hence, by applying
  Theorem~\ref{main}  and Proposition~\ref{operator-1} with $\gamma=\zeta=3/2$, we obtain
\begin{align*}
\|f^\dagger-p_m^\beta\|_2&\leq c_0^{-1}\Vert Ff^\dagger - Fp_m^\beta\rVert_{H^{\frac32}(S^d)}\\
& \leq c_0^{-1}\sqrt{\frac{1
  +\kappa}{\omega-1}}\, \Vert Ff^\dagger
\rVert_{H^{\omega+\frac32}(S^2)}  (1+m(m+1))^{-\frac{\omega-1}2}\notag +c_0^{-1}
\sqrt{\kappa} \beta(1+m(m+1))^{\frac{3}4}\\
& \leq \frac{3^{\frac34}}2 (\sin\vartheta_0)^{\frac12} c_0^{-1}\sqrt{\frac{1
  +\kappa}{\omega-1}}\, \Vert f^\dagger
\rVert_{H^{\omega}(S^2)}  (1+m(m+1))^{-\frac{\omega-1}2}\notag +
c_0^{-1}\sqrt{\kappa} \beta(1+m(m+1))^{\frac{3}4}.
\end{align*}

\subsection{ Planck and the Cosmic Microwave Background} \label{sub:planck} Here, the aim is to  estimate the full sky intensity map (displaying the intensity or temperature of every astronomical radio source across the celestial sphere) from a point source sky model \cite{tesi}. The data consists of noisy directional samples of the map, convolved with a model of the point spread function of the radio telescope (the beamshape of the dish antenna). One wishes to estimate the intensity field $f^\dagger$ of the stars on the celestial sphere using samples obtained by steering the dish antenna towards $N$ directions. Similarly as before, we assume $f^\dagger$ to be an element of the Sobolev space $H^\omega(S^2)$ for some $\omega>1$. The proposed point spread function is 
\[
h(x) = g(\rho(x,o)),
\]
where
\begin{equation}\label{eq:g-planck}
g(\vartheta)=\left(\frac{\lambda_0J_1\left(4\pi R\sin\Big(\dfrac{\vartheta}2\Big)\right)}{2R\sin\Big(\dfrac{\vartheta}2\Big)}\right)^2,
\end{equation}
and $J_1$ is the Bessel function of the first kind and order 1, corresponding to the far-field beamshape of an ideal circular aperture with radius $R\,\, \text {m}$ operating at a wavelength $\lambda_0\,\, \text{m}$. As usual, $\rho(x,o)$ is the angular distance of $x$ from the north pole $o$. 

\begin{figure}
     \centering
     \begin{subfigure}[b]{0.45\textwidth}
         \centering
         \includegraphics[width=\textwidth]{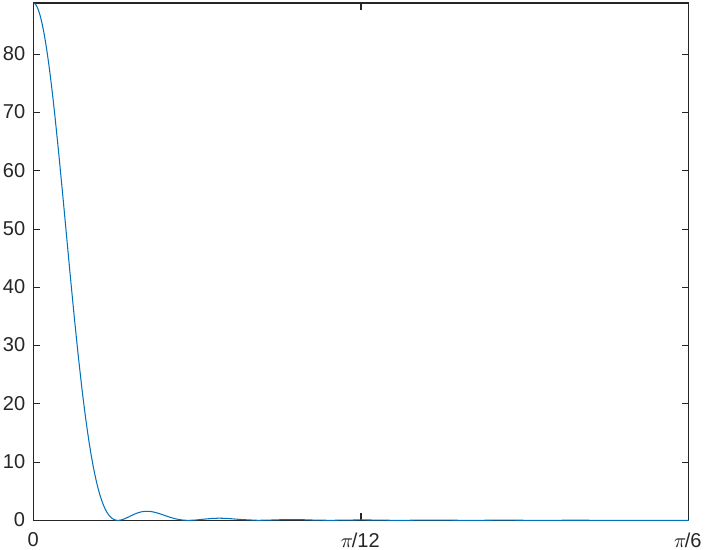}
         \caption{$\lambda_0=3$, $R=9$}
     \end{subfigure}
     \hfill
     \begin{subfigure}[b]{0.45\textwidth}
         \centering
         \includegraphics[width=\textwidth]{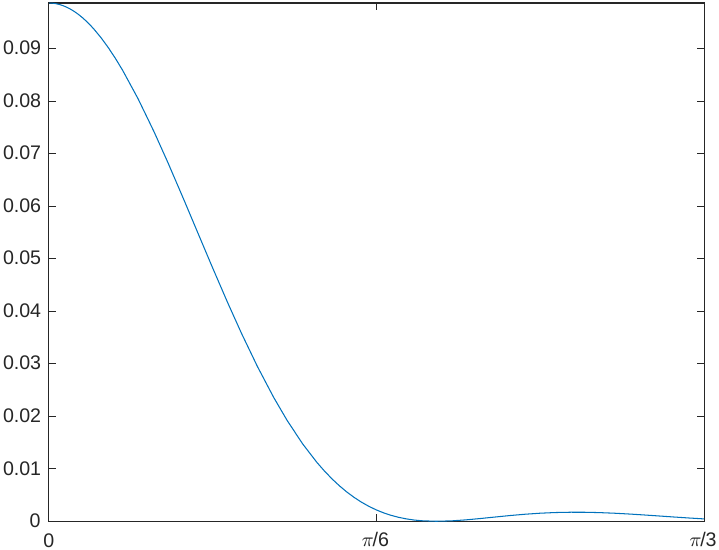}
         \caption{$\lambda_0=0.1$, $R=1$}
     \end{subfigure}
        \caption{The function $g$ in \eqref{eq:g-planck}, used to construct the convolution filter of the example from $\S$\ref{sub:planck}.}
        \label{fig:planck}
\end{figure}

In \cite{tesi} two possibilities are studied, one with $\lambda_0=3$ and $R=9$, where $N=768$ samples are used, and the second with $R=1$, $\lambda_0=0.1$ and $N=9248$ samples corresponding to data taken from the Planck radio telescope. See Figure~\ref{fig:planck} for an illustration of the function $g$ in these two cases.

Since the function $g$  can be extended to $\mathbb R$ as an even, $ C^\infty$, $2\pi$-periodic function,  $h$ belongs to all Sobolev spaces $H^{\sigma}(S^2)$, $\sigma\ge 0$. For the sake of simplicity, we study the case $\sigma=0$, so that for $m\ge1$ (see Remark~\ref{bmsfera}),
\[
|b_m|\le \frac{\|h\|_2}{(1+m(m+1))^{\frac14}}.
\]
We can apply  Theorem~\ref{main}
  with  $\gamma=1/2$ and, say, $\zeta=0$ to obtain
  \[
\|Fp_m^\beta-Ff^\dagger\|_2\le \sqrt{1+\kappa}\|h\|_2\|f^\dagger\|_{H^\omega(\mathcal B, \mathcal E)}\frac{1}{\Big(\omega-\dfrac12\Big)^{\frac12}}\frac1{(1+m(m+1))^{\frac{\omega-1/2}2}}+\sqrt{\kappa}\beta,
\]
where
\[
\|h\|_2=\left(\int_0^\pi|g(\vartheta)|^2\sin\Big(\frac\vartheta2\Big)\cos\Big(\frac\theta 2\Big)d\vartheta\right)^{\frac12}
=
\left(2\int_0^1\left|\frac{\lambda_0J_1(4\pi Rx)}{2Rx}\right|^4xdx\right)^{\frac12}.
\]
When $\lambda_0=3$ and $R=9$ one gets
$\|h\|_2\approx 1.064$, while when $\lambda_0=0.1$ and $R=1$, one gets $\|h\|_2\approx 0.0106$.

Unfortunately, a lower bound of the form \eqref{below} does not hold because $h$ is smooth and so $(b_m)_m$ has rapid decay; thus we cannot state an explicit bound for $\|p_m^\beta-f^\dagger\|_2$.

\subsection{ Lunar elemental abundance maps.} In this example, the aim is to build global distribution maps of radioactive elements on the surface of the Moon using real data collected by NASA's Lunar Prospector probe \cite{tesi}. These maps are used by scientists to retrace the Moon's geologic history. The emission of gamma rays at the surface of the Moon is modeled by an intensity function $f^\dagger$ belonging to some space $H^\omega(S^2)$, with $\omega>1$. The problem here consists in recovering $f^\dagger$ from a finite number of evaluations of $Ff^\dagger=h\ast f^\dagger$, where $h$ is the point spread function of the orbital gamma ray spectrometer, which is well fitted by
\[
h(x)=\left(1+\frac{R^2(\rho(x,o))^2}{2\sigma(t)^2}\right)^{-\iota(t)-1}
\]
where $R=1737.1\, \text{km}$ is the radius of the Moon, $t=30\,\text{km}$ is the altitude of the spacecraft, and
\[
\sigma(t)=0.704 t+1.39\, \text{km},\qquad \iota(t)=-4.87\, \text{km}^{-1}\times 10^{-4}t+0.631.
\]
As in $\S$\ref{sub:planck}, we use that $h\in L^2(S^2)$, so that
\[
|b_m|\le\frac{\|h\|_2}{(1+m(m+1))^{\frac14}}
\]
and we can apply Theorem~\ref{main} with $\gamma=1/2$ and, say, $\zeta=0$ to obtain
\[
\|Fp_m^\beta-Ff^\dagger\|_2\le \sqrt{1+\kappa}\|h\|_2\|f^\dagger\|_{H^\omega(S^2)}\frac{1}{\Big(\omega-\dfrac12\Big)^{\frac12}}\frac1{(1+m(m+1))^{\frac{\omega-1/2}2}}+\sqrt{\kappa}\beta,
\]
where
\[
\|h\|_2=\left(\int_0^\pi\left(1+\frac{R^2\vartheta^2}{2\sigma(t)^2}\right)^{-2\iota(t)-2}\sin\Big(\frac\vartheta2\Big)\cos\Big(\frac\vartheta 2\Big)d\vartheta\right)^{\frac12}
\approx0.0061.
\]
As in $\S$\ref{sub:planck}, we do not derive a lower bound of the form \eqref{below}.

\section{Sampling}\label{sec:sampling}

In this section, we revisit some fairly recent results regarding the approximation of smooth functions from a finite number of pointwise evaluations. This collection of results begins with the work of K. Gr\"ochenig \cite{Gro}, concerning the approximation of a function in a properly defined Sobolev space with a certain ``polynomial'' obtained by least squares from the samples of the function on the nodes of a Marcinkiewicz-Zygmund family.
One of the main features of this work is that the simplicity of the arguments allows for a great generality. For example, the ambient space is just a topological space with a probability measure, and the ``polynomials'' are simply the finite linear combinations of the functions of any orthonormal system in $L^2$.

Soon after, Gr\"ochenig's results were revisited in more explicit situations: for example, W. Lu and H. Wang \cite{LW} studied the case of the sphere, and J. Li, Y. Ling, J. Geng and H. Wang \cite{LLG} the case of a compact Riemannian manifold. In these cases, the orthonormal system is given by the eigenfunctions of the Laplace-Beltrami operator and the Sobolev spaces are defined as usual. In these situations, it is possible to improve Gr\"ochenig's results, obtaining asymptotically optimal estimates on
the decay of the approximation error.

\subsection{Sampling with Marcinkiewicz-Zygmund families}\label{sub:sampling_wit_MZ}

Let $\mathcal M$ be a space with a topology (usually a compact space) and let $\mu$ be  a probability measure on $\mathcal M$. Let $L^2(\mathcal M,\mu)$ denote the Hilbert space of (equivalence classes of) square-integrable functions, equipped with
\[
\langle f,g\rangle_2 = \int_{\mathcal M} f\,\overline g\,d\mu,\qquad \|f\|_2^2= \int_{\mathcal M} |f|^2\,d\mu.
\]
Futhermore, we assume that
\begin{enumerate}
\item[(i)] $\mathcal S=\{\psi_k:k\in\mathbb N\}$ is an orthonormal system in $L^2(\mathcal M,\mu)$,  such that each $\psi_k$ is continuous and bounded on $\mathcal M$, and $\psi_0\equiv 1$. We set $H(\mathcal S)=\overline{\mathrm{span}\{\psi_k:k\in\mathbb N\}}$.
\item[(ii)]$\mathcal L=\{\Lambda_k:\Lambda_k\ge0\}$ is a non-decreasing sequence with $\Lambda_k\to +\infty$.
\end{enumerate}
Let the Sobolev space associated to $\mathcal S$ and $\mathcal L$ be defined as
\[
H^\sigma( \mathcal S, \mathcal L)=\left\{f\in H( \mathcal S):\|f\|_{H^\sigma(\mathcal S,\mathcal L)}<+\infty\right\},\qquad \sigma\ge 0,
\]
where the norm is given by
\begin{equation}\label{eq:sobolev_norm}
\|f\|_{H^\sigma(\mathcal S,\mathcal L)}=\left(\sum_{k=0}^{+\infty}\left|\langle f,\psi_k\rangle_2\right|^2(1+\Lambda_k^2)^\sigma\right)^{\frac12}.
\end{equation}
Notice that $H^\sigma(\mathcal S,\mathcal L)$ is a Hilbert space, with the scalar product
\[
\langle f,g\rangle_{H^\sigma(\mathcal S,\mathcal L)}=\sum_{k=0}^{+\infty}\langle f,\psi_k\rangle_2\overline{\langle g,\psi_k\rangle_2}(1+\Lambda_k^2)^\sigma.
\]
Let $\mathcal Q_r(\mathcal S,\mathcal L)$ be the set of ``polynomials'' of degree $r$ on $\mathcal M$, that is
\[
\mathcal Q_r(\mathcal S,\mathcal L)=\operatorname{span}\{\psi_k:\Lambda_k\le r\}\subseteq L^2(\mathcal M).
\]

\begin{definition}
A doubly-indexed set $\{x_{m,j}:m\in\mathbb N,\,j=1,\ldots, L_m\}\subseteq{\mathcal M}$,
 a set of positive weights $\{\tau_{m,j}:m\in\mathbb N,\,j=1,\ldots, L_m\}$, and a non-negative, increasing sequence $\{r_m\}$ of ``degrees'' diverging to $+\infty$, such that 
\begin{equation}\label{MZ1}
A\|q\|_2^2\le\sum_{j=1}^{L_m}|q(x_{m,j})|^2\tau_{m,j}\le B\|q\|_2^2\qquad \text{ for all } q\in\mathcal Q_{r_m}(\mathcal S,\mathcal L),
\end{equation}
with constants $A,\,B>0$ independent of $m$, is called a Marcinkiewicz-Zygmund family $\mathcal X$ 
for $\mathcal S$ and $\mathcal L$. We set $\kappa=B/A$.    
\end{definition}

We define the series
\[
C_\sigma(\mathcal S,\mathcal L)=\sup_{x\in\mathcal M}\left(\sum_{k=0}^{+\infty}|\psi_k(x)|^2(1+\Lambda_k^2)^{-\sigma}\right)^{\frac12}
\]
and the 
remainders
\[
\phi_\sigma(r, \mathcal S, \mathcal L)=\sup_{x\in\mathcal M}\left(\sum_{k:\Lambda_k>r}|\psi_k(x)|^2(1+\Lambda_k^2)^{-\sigma}\right)^{\frac12}.
\]
When $C_\sigma(\mathcal S,\mathcal L)<+\infty$, then $\phi_\sigma(r,\mathcal S,\mathcal L)<+\infty$ too. If $\phi_\sigma(r, \mathcal S, \mathcal L)\to0$ as $r\to+\infty$, then $C_\sigma(\mathcal S,\mathcal L)<+\infty$. 

\begin{definition}
Let $\mathcal X$ be a Marcinkiewicz-Zygmund family 
for $\mathcal S$ and $\mathcal L$, with nodes $\{x_{m,j}\}$, weights $\{\tau_{m,j}\}$ and degrees $\{r_m\}$. For any finite sequence of complex numbers $y_m=\{y_{m,j}\}_{j=1}^{L_m}$,  we set
\begin{equation}\label{eq:T_dagger}
T^\dagger y_m=\operatornamewithlimits{argmin}_{q\in\mathcal Q_{r_m}(\mathcal S,\mathcal L)}\,\sum_{j=1}^{L_m}|y_{m,j}-q(x_{m,j})|^2\,\tau_{m,j}.
\end{equation}
\end{definition}

The following result shows that this minimization problem has indeed a unique minimizer, and provides a useful bound on its norm.

\begin{lemma}\label{dagger}
For any finite sequence of complex numbers $y_m=\{y_{m,j}\}_{j=1}^{L_m}$,  problem \eqref{eq:T_dagger} admits a unique minimizer  $T^\dagger y_m\in \mathcal Q_{r_m}(\mathcal S,\mathcal L)$, and
\[
\|T^\dagger y_m\|_2\le A^{-\frac12}\left(\sum_{j=1}^{L_m}|y_{m,j}|^2\tau_{m,j}\right)^{\frac12}.
\]
\end{lemma}

\begin{proof}
    Let us equip $\mathbb C^{L_m}$ with the weighted norm
    \[
    \|a\|_{\mathbb C^{L_m}}^2:=\sum_{j=1}^{L_m}|a_j|^2\tau_{m,j}.
    \]
    Let  $T\colon\mathcal Q_{r_m}(\mathcal S,\mathcal L)\to\mathbb C^{L_m}$ be the sampling operator defined by
    \(
    Tq=\left\{q(x_{m,j})\right\}_j.
    \)
    By \eqref{MZ1},
\begin{equation}
    A^{\frac12}\|q\|_2\le\|Tq\|_{\mathbb C^{L_m}}\le B^{\frac12}\|q\|_2,\label{bound}
\end{equation}
and $T$ is therefore injective. Let $\mathcal W=\operatorname{Im}T$ and $\Pi_{\mathcal W}$ be the orthogonal projection from $\mathbb C^{L_m}$ onto ${\mathcal W}$, then $T$ is an invertible operator from $\mathcal Q_{r_m}(\mathcal S,\mathcal L)$ onto $ \mathcal W$ with inverse $S\colon \mathcal W \to \mathcal Q_{r_m}(\mathcal S,\mathcal L)$. The bound \eqref{bound} implies that 
\[A^{\frac{1}{2}}\|S y\|_{2}\leq \|y\|_{\mathbb C^{L_m}} \qquad \forall y\in \mathcal W,\]
so that $\|S\|\leq A^{-\frac{1}{2}}$.
Furthermore, since
\[
\operatornamewithlimits{argmin}_{q\in \mathcal Q_{r_m}(\mathcal S,\mathcal L)}\|Tq- y_m\|^2_{\mathbb C^{L_m}}= \operatornamewithlimits{argmin}_{q\in \mathcal Q_{r_m}(\mathcal S,\mathcal L)} \|Tq- \Pi_{\mathcal W}y_m\|^2_{\mathbb C^{L_m}}= S\Pi_{\mathcal W}y_m,
\]
it follows that  the pseudoinverse of $T$ (see \cite{Penrose}) is $T^\dagger=S \Pi_{\mathcal W}$ and $\|T^\dagger\|\le \|S\| \|\Pi_{\mathcal W}\|\leq A^{-\frac{1}{2}}$.

\end{proof}

The following theorem holds (see 
 \cite[Theorem 3.3]{Gro}). It concerns the reconstruction of a smooth function $f^\dagger\in H^\sigma(\mathcal S, \mathcal L)$ from its pointwise samples $\left\{f^\dagger(x_{m,j})\right\}_j$.

\begin{Theorem}\label{KHG}
Let $\mathcal X$  a Marcinkiewicz-Zygmund family 
for $\mathcal S$ and $\mathcal L$, with nodes $\{x_{m,j}\}$, weights $\{\tau_{m,j}\}$ and degrees $\{r_m\}$. 
Let $f^\dagger\in H^\zeta(\mathcal S,\mathcal L)$ be continuous for some $\zeta\ge0$, and define
\[
q_m=T^\dagger\{f^\dagger(x_{m,j})\}=\operatornamewithlimits{argmin}_{q\in\mathcal Q_{r_m}(\mathcal S,\mathcal L)}\sum_{j=1}^{L_m}|f^\dagger(x_{m,j})-q(x_{m,j})|^2\tau_{m,j}.
\] Then for every $\sigma\ge \zeta$ and $m\in\mathbb N$
\[
\|f^\dagger-q_m\|_{H^\zeta(\mathcal S, \mathcal L)}\le \sqrt{1+\kappa}\|f^\dagger\|_{H^\sigma(\mathcal S, \mathcal L)}\phi_{\sigma-\zeta}(r_m, \mathcal S, \mathcal L).
\]
\end{Theorem}

For the sake of completeness, and given that this version is slightly more general and with a somewhat refined constant than the original result in \cite{Gro}, we report here the elementary proof.

\begin{proof}
Let us call $P_{r_m}$ the orthogonal projection of $H(\mathcal S)$ onto $Q_{r_m}(\mathcal S,\mathcal L)$. Given that if $f^\dagger\in H^\zeta(\mathcal S, \mathcal L)$, then $f^\dagger-P_{r_m}f^\dagger$ remains orthogonal to $Q_{r_m}(\mathcal S,\mathcal L)$ in $H^\zeta(\mathcal S,\mathcal L)$, 
\[
\|f^\dagger-q_m\|^2_{H^\zeta(\mathcal S,\mathcal L)}=
\|f^\dagger-P_{r_m}f\|_{H^\zeta(\mathcal S,\mathcal L)}^2+\|P_{r_m}f^\dagger-q_m\|^2_{H^\zeta(\mathcal S,\mathcal L)}.
\]
Now, 
\begin{align*}
\|f^\dagger-P_{r_m}f^\dagger\|_{H^\zeta(\mathcal S,\mathcal L)}&=\left\|\sum_{\Lambda_k>r_m}\langle f^\dagger,\psi_k\rangle_2\psi_k\right\|_{H^\zeta(\mathcal S,\mathcal L)}
=\left\|\sum_{\Lambda_k>r_m}\langle f^\dagger,\psi_k\rangle_2(1+\Lambda_k^2)^{\frac\zeta2}\psi_k\right\|_{2}
\\
&\leq\left\|\sum_{\Lambda_k>r_m}\langle f^\dagger,\psi_k\rangle_2(1+\Lambda_k^2)^{\frac\sigma2}(1+\Lambda_k^2)^{-\frac{\sigma-\zeta}2}\psi_k\right\|_{\infty}\le \|f^\dagger\|_{H^\sigma(\mathcal S, \mathcal L)}\phi_{\sigma-\zeta}(r_m, \mathcal S, \mathcal L).
\end{align*}
On the other hand, using the notation of the proof of Lemma~\ref{dagger}, since $P_{r_m}f^\dagger\in Q_{r_m}(\mathcal S,\mathcal L)$,
\begin{align*}
\|P_{r_m}f^\dagger-q_m\|_{H^\zeta(\mathcal S,\mathcal L)}&\le
(1+r_m^2)^{\frac\zeta2}\|P_{r_m}f^\dagger-q_m\|_{2}\\
&=(1+r_m^2)^{\frac\zeta2}\|T^{\dagger}\{P_{r_m}f^\dagger(x_{m,j})\} -T^\dagger\{f^\dagger(x_{m,j})\}\|_{2}\\
&\le A^{-\frac12}(1+r_m^2)^{\frac\zeta2}\left(\sum_{j=1}^{L_m}|P_{r_m}f^\dagger(x_{m,j}) -f^\dagger(x_{m,j})|^2\tau_{m,j}\right)^{\frac12}\\
&\le B^{\frac12}A^{-\frac12}(1+r_m^2)^{\frac\zeta2}\sup_{x}|P_{r_m}f^\dagger(x) -f^\dagger(x)|\\
&= B^{\frac12}A^{-\frac12}(1+r_m^2)^{\frac\zeta2}\sup_{x}\left|\sum_{\Lambda_k>r_m}\langle f^\dagger,\psi_k\rangle_2\psi_k(x)\right|\\
&\le B^{\frac12}A^{-\frac12}(1+r_m^2)^{\frac\zeta2}\|f^\dagger\|_{H^\sigma(\mathcal S, \mathcal L)}\phi_{\sigma}(r_m, \mathcal S, \mathcal L)\\
&\le B^{\frac12}A^{-\frac12}\|f^\dagger\|_{H^\sigma(\mathcal S, \mathcal L)}\phi_{\sigma-\zeta}(r_m, \mathcal S, \mathcal L),
\end{align*}
where we have used that $\sum_{j=1}^{L_m}\tau_{m,j}\le B$
(simply set $q\equiv 1$ in \eqref{MZ1}).
\end{proof}

Notice that normally $\sum_{j=1}^{L_m}\tau_{m,j}=1\le B$, in which case   the factor $\kappa$ could be replaced by $A^{-\frac12}$.

 In the original theorem in \cite{Gro}, $\zeta$ is set equal to $0$ and $\mathcal S$ is assumed to be a complete orthonormal system, so that $H( \mathcal S)=L^2(\mathcal M)$. Also, in that paper, $r_m$ is set to be equal to $m$.

In a real-life context, evaluations $f^\dagger(x_{m,j})$ may be subject to certain measurement errors. The following result takes this possibility into account.
\begin{Corollary}\label{noise}
Let $\mathcal X$ be a Marcinkiewicz-Zygmund family for $\mathcal S$ and $\mathcal L$, with nodes $\{x_{m,j}\}$, weights $\{\tau_{m,j}\}$ and degrees $\{r_m\}$. 
Let $f^\dagger$ be a continuous function in $H^\zeta(\mathcal S,\mathcal L)$, $\zeta\ge0$, and let $\{y_{m,j}\}$ be the corresponding
family of noisy samples such that
\[
 |y_{m,j}- f^\dagger(x_{m,j})|\leq\beta, \qquad j=1,\ldots,L_m,
\]
where $\beta \ge 0$ is the noise level. Set 
\[
q_m^\beta=T^\dagger\{y_{m,j}\}=\operatornamewithlimits{argmin}_{q\in\mathcal Q_{r_m}(\mathcal S,\mathcal L)}\sum_{j=1}^{L_m}|y_{m,j}-q(x_{m,j})|^2\tau_{m,j}.
\]
Then for every $\sigma\ge \zeta$ and $m\in\mathbb N$
\[
\|f^\dagger-q_m^\beta\|_{H^\zeta(\mathcal S,\mathcal L)}\le \sqrt{1+\kappa}\|f^\dagger\|_{H^\sigma(\mathcal S, \mathcal L)}\phi_{\sigma-\zeta}(r_m, \mathcal S, \mathcal L)+\sqrt{\kappa}\beta(1+r_m^2)^{\frac\zeta2}.
\]
\end{Corollary}
 \begin{proof}
By Lemma~\ref{dagger} and Theorem~\ref{KHG} we have
\begin{align*}
&\|f^\dagger-T^\dagger \{y_{m,j}\}\|_{H^\zeta(\mathcal S,\mathcal L)}\\
&\le \|f^\dagger-T^\dagger \{f^\dagger(x_{m,j})\}\|_{H^\zeta(\mathcal S,\mathcal L)}+\|T^\dagger \{f^\dagger(x_{m,j})\}-T^\dagger \{y_{m,j}\}\|_{H^\zeta(\mathcal S,\mathcal L)}\\
&\le \|f^\dagger-T^\dagger \{f^\dagger(x_{m,j})\}\|_{H^\zeta(\mathcal S,\mathcal L)}+(1+r_m^2)^{\frac\zeta2}\|T^\dagger \{f^\dagger(x_{m,j})\}-T^\dagger \{y_{m,j}\}\|_{2}\\
&\le \sqrt{1+\kappa}\|f^\dagger\|_{H^\sigma(\mathcal S, \mathcal L)}\phi_{\sigma-\zeta}(r_m, \mathcal S, \mathcal L)+A^{-\frac12}(1+r_m^2)^{\frac \zeta2}\left(\sum_{k=1}^{L_m}|f^\dagger(x_{m,j})-y_{m,j}|^2\tau_{m,j}\right)^{\frac12}\\
&\le \sqrt{1+\kappa}\|f^\dagger\|_{H^\sigma(\mathcal S, \mathcal L)}\phi_\sigma(r_m, \mathcal S, \mathcal L)+B^{\frac12}A^{-\frac12}\beta(1+r_m^2)^{\frac\zeta2},
\end{align*}
where we have used that $\sum_{j=1}^{L_m}\tau_{m,j}\le B$.
\end{proof}

\medskip

The following result gives a lower estimate for the function $\phi_\sigma(r_m,\mathcal S,\mathcal L)$.

\begin{Proposition}
\label{lower}
Assume that for a doubly indexed collection of nodes $\{x_{m,j}\}_{m=1,j=1}^{+\infty,L_m}$, for corresponding weights $\{\tau_{m,j}\}$, for an increasing sequence of degrees $\{r_m\}_{m=1}^{+\infty}$ diverging to $\infty$, and for $\sigma\ge0$ and $\zeta\geq0$,
there exists a function $\Psi_{\sigma,\zeta}(r_m)$ such that
for every polynomial $f$ (finite linear combination of the $\psi_k$'s), if we define
\[
q_m=\operatornamewithlimits{argmin}_{q\in\mathcal Q_{r_m}(\mathcal S,\mathcal L)}\sum_{j=1}^{L_m}|f(x_{m,j})-q(x_{m,j})|^2\tau_{m,j},
\] there holds
\[
\|f-q_m\|_{H^\zeta(\mathcal S, \mathcal L)}\le \|f\|_{H^\sigma(\mathcal S, \mathcal L)}\Psi_{\sigma,\zeta}(r_m).
\]
Then, if we let $k_m$  be the smallest value  $k$ for which $\Lambda_k>r_m$, we have
\[
\Psi_{\sigma,\zeta}(r_m)\ge(1+\Lambda_{k_m}^2)^{-\frac{\sigma-\zeta}{2}}.
\]
\end{Proposition}
\begin{proof}
Set $f=\psi_{k_m}$. Then $f$ is orthogonal to $\mathcal Q_{r_m}(\mathcal S,\mathcal L)$, and so
\begin{align*}
(1+\Lambda_{k_m}^2)^{\zeta}&=\|f\|^2_{H^\zeta(\mathcal S, \mathcal L)}\le\|f\|_{H^\zeta(\mathcal S, \mathcal L)}^2 + \|-q_m\|^2_{H^\zeta(\mathcal S, \mathcal L)}
=\|f-q_m\|_{H^\zeta(\mathcal S, \mathcal L)}^2\\
&\le
\|f\|^2_{H^\sigma(\mathcal S, \mathcal L)}(\Psi_{\sigma,\zeta}(r_m))^2
=(1+\Lambda_{k_m}^2)^\sigma(\Psi_{\sigma,\zeta}(r_m))^2,
\end{align*}
as claimed.
\end{proof}

Assume now that $\mathcal M$ is a compact Riemannian manifold without boundary. Let $\mathcal S$ denote the orthonormal basis of eigenfunctions of the Laplace-Beltrami operator and $\mathcal L$ the set of the square roots of the corresponding eigenvalues. It has been proved in \cite{LLG} (see also \cite{LW}) that in this case, under the hypotheses of Theorem~\ref{KHG} with $\zeta=0$,
\[\|f^\dagger-q_m\|_2\leq C(1+\kappa^\frac12)r_m^{-\sigma}\|f^\dagger\|_{H^\sigma(\mathcal S,\mathcal L)},\]
which is asymptotically sharp in view of Proposition~\ref{lower}. In our context, we require the possibility for $\mathcal S$ to be an orthonormal system but not necessarily a basis. 

It is unclear whether the above result from \cite{LLG} can be extended to a proper subset of a basis of eigenfunctions of the Laplace–Beltrami operator. For this reason, we shall employ a slight refinement of Gr\"ochenig's result (Theorem \ref{KHG}), which yields a better constant, though it generally does not achieve the optimal asymptotic decay.

\subsection{The construction of Marcinkiewicz-Zygmund families in Riemannian manifolds}\label{MZ}

As a first and remarkable example, let us suppose that $\mathcal{M}$ is a connected compact orientable $d$-dimensional Riemannian
manifold without boundary and with normalized Riemannian measure $d\mu$, such that
$\mu\left(  \mathcal{M}\right)  =1$.   Let us consider the distinct eigenvalues $\{0=\lambda_{0}^{2}<\lambda_{1}
^{2}<\lambda_{2}^{2}<\ldots\}$ of the 
Laplace-Beltrami operator and let $\mathcal B=\left\{  \varphi
_{m}^{\ell}\right\}  _{m=0,\ell=1}^{+\infty,\delta_m}$ be a corresponding orthonormal basis of eigenfunctions, so that  $\Delta\varphi_{m}^{\ell}=\lambda_{m}^{2}\varphi_{m}^{\ell}$. Set $\mathcal E=\{\lambda_m\}_{m=0, \ell=1}^{+\infty,\delta_m}$.

In this case, the space of diffusion polynomials of bandwith $r\geq0$ is
\[
\mathcal Q_{r}(\mathcal B,\mathcal E)=\operatorname*{span}\left\{  \varphi_{m}^\ell:\lambda_{m}\leq r\right\}  .
\]
By Weyl's estimates on the spectrum of an elliptic operator \cite[Theorem~17.5.3 and Corollary~17.5.8]{HOR}, 
\begin{equation}
\label{spectrum}
\dim\left(  \mathcal Q_{r}(\mathcal B,\mathcal E)\right)  \asymp r^{d} \quad \text{ and } \quad \sum_{m:\lambda_m\le r}\sum_{\ell=1}^{\delta_m}|\varphi_m^\ell(x)|^2\le W r^d
\end{equation}
for some $W>0$.

We begin with a definition.

\begin{definition}\label{def:partition}
Let $0<a\leq1\leq b$ and $0<c_1<c_2$. We say that a collection $\mathcal{R}=\{R_{j} \}_{j=1}^{N}$ of measurable subsets of $\mathcal M$ is a partition of $\mathcal M$ with parameters $a$, $b$, $c_1$ and $c_2$
if the following conditions hold:
\begin{itemize}
\item $\cup_{j=1}^N R_j=\mathcal M$ and $\mu(R_i\cap R_j)=0$ for every $1\le i<j\le N$;
\item $a/N\leq\mu(R_{j})\leq b/N$ for every $j=1, \cdots N$;
\item and each $R_{j}$ is contained in a geodesic ball of radius $ c_2 N^{-1/d} $ and contains a geodesic ball of radius $c_1 N^{-1/d}$.
\end{itemize}
\end{definition}

Maggioni, Mhaskar and Filbir \cite{FM1,FM, maggioni} have proven that any choice of $N$ points, each one in a region $R_j$ as above, produces a Marcinkiewicz-Zygmund family with weights equal to the measures of the regions.

\begin{Theorem}[{\cite[Theorem 5.1]{FM}}]
\label{MZ polys} Let $0<a\leq1\leq b$ and $0<c_1<c_2$. 
Then, there exist positive constants $C$ and $D$ such that for all  $N\ge 1$ and  $r>0$ such that $N\ge C r^d$, for all partitions $\{ R_{j}\}_{j=1}^{N}$ with parameters $a$, $b$, $c_1$ and $c_2$ as in Definition \ref{def:partition},  for all $x_{j} \in R_{j}$, for all $q \in \mathcal Q_{r}(\mathcal B,\mathcal{E})$ it holds
\[
\left\vert \int_{\mathcal{M}}\left\vert q(  x)  \right\vert^2
d\mu(  x)  -\sum_{j=1}^{N}\mu(R_j)\left\vert q(
x_{j})  \right\vert^2\right\vert \leq DrN^{-1/d}\int_{\mathcal{M}
}\left\vert q(  x)  \right\vert^2 d\mu(  x)  .
\]
In particular, for any $\varepsilon>0$, if $N\ge \max\{C, D^d/\varepsilon^d\}r^d$, then 
\[
\left\vert \int_{\mathcal{M}}\left\vert q(  x)  \right\vert^2
d\mu(  x)  -\sum_{j=1}^{N}\mu(R_j)\left\vert q(
x_{j})  \right\vert^2\right\vert \leq \varepsilon\int_{\mathcal{M}
}\left\vert q(  x)  \right\vert^2 d\mu(  x)  ,
\]
that is
\[
(1-\varepsilon) \int_{\mathcal{M}}\left\vert q(  x)  \right\vert^2
d\mu(  x)  \le \sum_{j=1}^{N}\mu(R_j)\left\vert q(
x_{j})  \right\vert^2 \leq (1+\varepsilon)\int_{\mathcal{M}%
}\left\vert q(  x)  \right\vert^2 d\mu(  x)  .
\]
\end{Theorem}

It remains to show that partitions as required by the previous theorem do exist. Indeed the following theorem holds (see \cite{cubature, design, GL}).

\begin{Theorem} \label{partition}
There exist two constants $0<c_3<c_4$ such that for all constants $a$ and $b$ with $0<a\leq 1\leq b$, for every $N \geq1$ and for every choice of weights $\{\tau_{j} \}_{j=1}^{N}$ with $\sum_{j=1}^{N} \tau_{j}=1$ and $a/N \leq \tau_{j} \leq b/N$,  
there is a partition of $\mathcal M$, $\mathcal{R}=\{R_{j} \}_{j=1}^{N}$, with parameters $a$, $b$, $c_3(a^2/b)^{1/d}$ and $c_4b^{1/d}$ such that $\mu(R_{j})=\tau_j$ for all $j=1,\ldots,N$.
\end{Theorem}

\section{Sampling in inverse problems}\label{sec:sampling_IP}

In this section, we apply the previous results to inverse problems. More precisely, we consider the problem consisting of the recovery of a signal $f^\dagger$ from sampled measurements of $Ff^\dagger$, where $F$ is a linear map.

We make the following assumptions. The first assumption concerns the measurement space $Y$ for the inverse problem, which will be a Sobolev space on $\mathcal{M}$.

\begin{assumption}\label{spazio Y}
Let $\mathcal M$ be a topological space and $\mu$ be a probability measure in $\mathcal M$. Let $\mathcal B=\{\varphi_k:k\in\mathbb N\}$ be an orthonormal basis in $L^2(\mathcal M,\mu)$, such that each $\varphi_k$ is continuous and bounded on $\mathcal M$, and $\varphi_0\equiv 1$. 
Let $\mathcal E=\{\lambda_k:\lambda_k\ge0\}$ be a non-decreasing sequence with $\lambda_k\to +\infty$.
Assume that there exists a Marcinkiewicz-Zygmund family $\mathcal X$ for $\mathcal B$ and $\mathcal E$, with nodes $\{x_{m,j}:m\in\mathbb N,\,j=1,\ldots, L_m\}\subseteq{\mathcal M}$,
weights $\{\tau_{m,j}\}$ and degrees $\{r_m\}$, such that 
\[
A\|q\|_2^2\le\sum_{j=1}^{L_m}|q(x_{m,j})|^2\tau_{n,j}\le B\|q\|_2^2\qquad \text{ for all } q\in\mathcal Q_{r_m}(\mathcal B,\mathcal E),
\]
with constants $A,\,B>0$ independent of $m$.
Let $Y=H^\sigma(\mathcal B,\mathcal E)$, for some $\sigma\ge0$.

\end{assumption}

In the second assumption, we fix a Galerkin nested sequence of subspaces of $X$, the space where the unknown $f$ belongs to.

\begin{assumption}\label{spazio X}
Let $X$ be a Banach space and
assume that $X$ admits a multi-resolution structure, that is a family $\{G_m\}_{m\in \mathbb N}$ of subspaces of $X$ such that
\begin{enumerate}[(i)]
\item each $G_m$ is finite-dimensional,
\item $G_m\subseteq G_{m+1}$,
\item $\overline{\cup_{m=0}^{+\infty}G_m}=X$. 
\end{enumerate} 
\end{assumption}

The last assumption quantifies the smoothing effect of $F$.

\begin{assumption}\label{operatore F}
Let $F\colon X\to Y$ be linear and bounded, and assume that 
\begin{enumerate}[(i)]
\item$
F(G_m)\subseteq \mathcal Q_{r_{m}}(\mathcal B,\mathcal E)$, 
\item
$1\in F(G_0)$,
\item $\dim \operatorname{Im} F=+\infty$.
\end{enumerate}
\end{assumption}

We now build a new orthonormal system $\mathcal S$ of $L^2(\mathcal M, \mu)$ and a new non-decreasing sequence $\mathcal L=\{\Lambda_k:\Lambda_k\ge0\}$  with $\Lambda_k\to +\infty$, adapted to the operator $F$. Let us begin with $\mathcal S_0= \{\psi_k:0\le k\le d_0\}$, an orthonormal basis for $F(G_0)$, with $\psi_0\equiv 1$. Being $F(G_0)\subseteq \mathcal Q_{r_{0}}(\mathcal B,\mathcal E)$, for each $k$ between $0$ and $d_0$ we set $\Lambda_k=r_{0}$.
Let us now complete (if necessary)  $\mathcal S_0$ to an orthonormal basis $\mathcal S_1= \{\psi_k:0\le k\le d_1\}$ for $F(G_1)$. Being $F(G_1)\subseteq \mathcal Q_{r_{1}}({\mathcal{B}},\mathcal E)$, for each $k$ between $d_0+1$ and $d_1$ (if any) we set $\Lambda_k=r_{1}$.
Proceeding similarly, for any $m\ge 1$, we complete (if necessary)  $\mathcal S_m$ to an orthonormal basis $\mathcal S_{m+1}= \{\psi_k:0\le k\le d_{m+1}\}$ for $F(G_{m+1})$. 
Being $F(G_{m+1})\subseteq \mathcal Q_{r_{{m+1}}}(\mathcal B,\mathcal E)$, 
for each $k$ between $d_m+1$ and $d_{m+1}$ (if any) we set $\Lambda_k=r_{{m+1}}$.
We call $\mathcal S=\{\psi_k:k\in\mathbb N\}$ and $\mathcal L=\{\Lambda_k:k\in\mathbb N\}$. Note that, thanks to Assumption~\ref{operatore F} (iii), we have $\Lambda_k\to+\infty$.

Furthermore, $\mathcal Q_{r_{m}}(\mathcal S,\mathcal L)=\mathrm{span}\{\psi_k :\Lambda_k\le r_{m}\}=\mathrm{span}\{\psi_k:0\le k\le d_{m}\}=F(G_m)\subseteq \mathcal{Q}_{r_m}(\mathcal B, \mathcal E)$.

In the following result, we investigate the recovery of $Ff^\dagger$ from a finite number of noisy sampled measurements.

\begin{Corollary}\label{noisebis}
Let us suppose that Assumptions \ref{spazio Y}, \ref{spazio X}, \ref{operatore F} hold and 
let $\sigma\ge \zeta\ge 0$. Let $f^\dagger\in X$ be such that $Ff^\dagger$ is continuous and let $\{y_{m,j}\}$ be such that
\[
 |y_{m,j}- Ff^\dagger(x_{m,j})|\leq\beta \quad j=1,\ldots,L_m
\]
where $\beta \ge 0$ is the noise level. Let 
\[
p_m^\beta\in\operatornamewithlimits{argmin}_{p\in G_m}\sum_{j=1}^{L_{m}}|y_{m,j}-Fp(x_{m,j})|^2\tau_{m,j}.
\] Then
\[
\|Ff^\dagger-Fp_m^\beta\|_{H^\zeta(\mathcal S,\mathcal L)}\le \sqrt{1+\kappa}\|Ff^\dagger\|_{H^\sigma(\mathcal S, \mathcal L)}\phi_{\sigma-\zeta}(r_{m}, \mathcal S, \mathcal L)+\sqrt{\kappa}\beta(1+r_m^2)^{\frac\zeta2}.
\]
\end{Corollary}

\begin{proof}
Since $F(G_m)=\mathcal Q_{r_{m}}(\mathcal S,\mathcal L)$, it follows that
\[Fp_m^\beta=\operatornamewithlimits{argmin}_{q\in \mathcal{Q}_{r_{m}}(\mathcal S, \mathcal L)}\sum_{j=1}^{L_{m}}|y_{m,j}-q(x_{m,j})|^2\tau_{m,j}.\]
Since $\mathcal Q_{r_{m}}(\mathcal S,\mathcal L)\subseteq \mathcal Q_{r_{m}}(\mathcal B,\mathcal E)$ by Assumption \ref{operatore F}, it follows that $\mathcal X$ is a Marcinkiewicz-Zygmund family also for $\mathcal S$ and $\mathcal L$ with the same constants $A$ and $B$, and by Corollary \ref{noise}, the thesis follows.
\end{proof}

\section{Convolutions on compact two-point homogeneous spaces}\label{sec:convolution}

Here, we apply the results of Section~\ref{sec:sampling_IP} to a particular class of forward maps $F$, namely, convolutions on compact two-point homogeneous spaces. Two-point homogeneous spaces are introduced in $\S$\ref{sub:two-point}, zonal functions and convolutions are described in $\S$\ref{sub:zonal} and, finally, the problem of sampling convolutions is discussed in $\S$\ref{sub:sampling_convolutions}.

\subsection{Two-point homogeneous spaces}\label{sub:two-point}

A $d$-dimensional Riemannian manifold $\mathcal{M}$ with distance $\rho$ is
said to be a two-point homogeneus space if given four points $x_{1}%
,x_{2},y_{1},y_{2}\in\mathcal{M}$ such that $\rho(x_{1},y_{1})=\rho
(x_{2},y_{2})$, then there exists an isometry $g$ of $\mathcal{M}$ such that
$gx_{1}=x_{2}$ and $gy_{1}=y_{2}$. Wang in \cite{Wang} has completely
characterized compact connected two-point homogeneous spaces. More precisely
$\mathcal{M}$ is isometric to one of the following compact rank
$1$ symmetric spaces:

\begin{enumerate}[i)]
\item the Euclidean sphere $S^{d}=SO(d+1) /SO(d)\times\{1\}$,
$d\geqslant1$;

\item the real projective space $P^{n}( \mathbb{R}) =O(n+1) /O( n)
\times O( 1) $, $n\geqslant2$;

\item the complex projective space $P^{n}( \mathbb{C}) =U(n+1) /U( n)
\times U( 1) $, $n\geqslant2$;

\item the quaternionic projective space $P^{n}( \mathbb{H})=Sp(n+1)/Sp(
n) \times Sp( 1) $, $n\geqslant2$;

\item the octonionic projective plane $P^{2}( \mathbb{O}) $.
\end{enumerate}

In the following we will assume that $\mathcal{M}$ is one of the above
symmetric spaces and that $d$ is its real dimension. In particular the real
dimension of $P^{n}(\mathbb{K})$ is $d=nd_{0}$, where $d_{0}=1,2,4,8$
according to the real dimension of $\mathbb{K}=\mathbb{R}$, $\mathbb{C}$,
$\mathbb{H}$ and $\mathbb{O}$ respectively. In the case of $S^{d}$ it will be
convenient to set $d_{0}=d$. See \cite[pp. 176-178]{Gangolli}, see also
\cite{H}, \cite{Skriganov} and \cite{Wolf}. In the following, to keep the notation simple, we will use
\begin{equation}\label{eq:abdd0}
a =\frac{d-2}{2}, \ \quad\ b =\frac{d_{0}-2}{2}.
\end{equation}

Let $\mu$ be the Riemannian measure on $\mathcal{M}$ normalized so that
$\mu(\mathcal{M})=1$. Let us assume that $\rho$ is normalized so that $\mathrm{diam}(\mathcal{M})=\pi$ and let $B_{r}(x)=\{y\in\mathcal{M}:\rho(x,y)<r\}$.

If $o$ is a fixed point in $\mathcal{M}$, then $\mathcal{M}$ can be identified
with the homogeneous space $G/K$, where $G$ is the group of isometries of
$\mathcal{M}$ and $K$ is the stabilizer of $o$.  We will also identify
functions $f(x)$ on $\mathcal{M}$ with right $K$-invariant functions $\widetilde f(g)$ on
$G$ by setting $\widetilde f(g)=f(x)$ when $go=x$. The measure $\mu$ is invariant under the action of $G$;
in other words, for every $g\in G$,
\[
\int_{\mathcal{M}}f(gx)d\mu(x)=\int_{\mathcal{M}}f(x)d\mu(x).
\]

Let $dg$ and $dk$ denote  left Haar measures on $G$ and  $K$, respectively. These measures are normalized in such a way that $\int_G dg =1$ and $\int_K dk =1$. By \cite[Theorem 2.51]{Fo}, we have
\begin{equation}\label{eq:251}
\int_G f(go)\,dg =     \int_{\mathcal{M}}f(x)d\mu(x), \qquad f\in L^1(\mathcal M).
\end{equation}

\subsection{Zonal functions and convolutions}\label{sub:zonal}

\begin{definition}
A function $f$ on $\mathcal{M}$ is a zonal function if for every
$x\in\mathcal{M}$ and every $k\in K$ we have $f( kx) =f( x) $.
\end{definition}

\begin{lemma}
\label{Lemma misura}Let $f$ be a zonal function. Then $f( x) $ depends only on
$\rho( x,o) $. Furthermore, defining $f_{0}$ so that $f( x) =f_{0}( \rho( x,o)
) $ we have%
\begin{equation}
\int_{\mathcal{M}}f( x) d\mu( x) =\int_{0}^{\pi}f_{0}( r) A( r) dr,
\label{IntegraleZonale}%
\end{equation}
where%
\[
A( r) =c( a,b) \left(  \sin\Big(\frac{r}{2}\Big)\right)  ^{2a+1}\left(  \cos\Big(\frac{r}
{2}\Big)\right)  ^{2b+1}%
\]
and
\[
c(a,b) =\left(  \int_{0}^{\pi}\left(  \sin\Big(\frac{r}{2}\Big)\right)  ^{2a+1}\left(
\cos\Big(\frac{r}{2}\Big)\right)  ^{2b+1}dr\right)  ^{-1} =\frac{\Gamma(a+b+2)}
{\Gamma(a+1)\Gamma(b+1)}.
\]

\end{lemma}

\begin{proof}
Let $x,y\in\mathcal{M}$ be such that $\rho( x,o) =\rho( y,o) $. Since $\mathcal{M}$ is two-point homogeneous there exists $g\in G$ such that $gx=y$ and $go=o$. Thus, $g\in K$ and $f( y) =f( gx) =f( x) $. Equation \eqref{IntegraleZonale} follows from (4.17) in \cite{Gangolli}. 
\end{proof}
We now recall the definition of convolution between two functions $h$ and $f$ defined on the sphere. As shown by Lemma~\ref{zonal_par} below, the convolution only depends on the zonal function $\widetilde h$ associated to $h$ by~\eqref{eq:tilde_h}. 
\begin{definition}
If $f$ and $h$ are two integrable functions on $\mathcal M$, then we define the convolution of $h$ and $f$ by
\[
h\ast f(x)=\int_G h(g^{-1}x) f(go)dg.
\]
\end{definition}

The following lemma shows that this convolution with filter $h$ coincides with the convolution with a zonal function $\widetilde h$.
\begin{lemma}\label{zonal_par}
    Let $f,h\in L^1(\mathcal{M})$. Write
    \begin{equation}\label{eq:tilde_h}
    \widetilde h(x) :=\int_K h(kx)\,dk = \widetilde{h}_0(\rho(x, o)).
\end{equation}
    Then $ \widetilde h$ is zonal and
    \[
    h*f =  \widetilde h*f.
    \]
    Furthermore
       \[
\widetilde h*f(x)=\int_{\mathcal M}\widetilde{h}_0(\rho(x, y))f(y)d\mu(y).
    \]
\end{lemma}
\begin{proof}
Take $x\in \mathcal M$ and write $x=u o$ with some $u\in G$. Recall that $dg$ is also a right Haar measure on $G$. Observe that for any $k\in K$, by the change of variable $w=k^{-1}g^{-1}u$,
\[
h\ast f(x)=\int_G h(g^{-1}uo) f(go)dg=\int_G h(kwo) f(uw^{-1}k^{-1}o)dw
=\int_G h(kwo) f(uw^{-1}o)dw.
\]
Integration in $k\in K$ gives
\[
h\ast f(x)=\int_K \int_G h(kwo) f(uw^{-1}o)dwdk= \int_G f(uw^{-1}o) \int_K h(kwo)dkdw.
\]
Notice that the function 
$\widetilde h$ is zonal on $\mathcal M$.
It follows that 
\[
h\ast f(x)= \int_G f(uw^{-1}o) \widetilde{h}(wo)dw= \int_G \widetilde{h}(g^{-1}uo)f(go) dg=\widetilde{h}\ast f(x).
\]
In other words, the convolution depends only on $\widetilde{h}$. Observe that 
\begin{align*}
\widetilde{h}\ast f(x)&= \int_G \widetilde{h}(g^{-1}x)f(go)dg= \int_G \widetilde{h}_0(\rho(g^{-1}x, o))f(go)dg\\&= \int_G\widetilde{h}_0(\rho(x, go))f(go)dg=\int_{\mathcal M}\widetilde{h}_0(\rho(x, y))f(y)d\mu(y),
\end{align*}
where we used \eqref{eq:251} in the last identity.
\end{proof}

Let $\Delta$ be the Laplace-Beltrami operator on $\mathcal{M}$, let \[
0=\lambda_{0}^2<\lambda_{1}^2<
\lambda_{2}^2<\ldots,
\]
be the distinct eigenvalues of $\Delta$
arranged in increasing order, let $\mathcal{H}_{m}$ be the eigenspace associated with the eigenvalue $\lambda^2_{m}$ and let $\delta_m$ be its dimension. In the case $m=0$, we have $\mathcal{H}_0=\operatorname{span}\{1\}$ and $\delta_0=1$.
It is well known that%
\begin{equation}
L^{2}(\mathcal{M})=
{\displaystyle\bigoplus_{m=0}^{+\infty}}
\mathcal{H}_{m}. \label{Decomp L2}%
\end{equation}
If $f(x)=f_{0}(\rho(x,o))$ is a zonal function on $\mathcal{M}$, then%
\begin{equation}
\Delta f(x)=\left.  \frac{1}{A(t)}\frac{d}{dt}\left(  A(t)\frac{d}{dt}%
f_{0}(t)\right)  \right\vert _{t=\rho(x,o)} \label{Radial Laplace Beltrami}%
\end{equation}
(see (4.16) in \cite{Gangolli}).

Next, we study $\mathcal{H}_m$ as a reproducing kernel Hilbert space.

\begin{definition}
The zonal spherical function of degree $m\in\mathbb{N}$ with pole
$x\in\mathcal{M}$ is the unique function $Z_{x}^{m}\in\mathcal{H}_{m}$, given
by the Riesz representation theorem, such that for every $Y\in\mathcal{H}_{m}$%
\[
Y( x) =\int_{\mathcal{M}}Y( y) Z_{x}^{m}( y) d\mu( y) .
\]

\end{definition}

The next lemma summarizes the main properties of zonal functions and is
essentially taken from \cite{SW}, where the case $\mathcal{M}=S^{d}$ is
discussed in detail. For a proof of the general case, see \cite{BGGM1, BGG}.

\begin{lemma}\label{lem:zonal_rkhs}
Take $m\in \mathbb N$ and $x,y\in\mathcal{M}$. We have that:
\begin{enumerate}[(i)]
\item \label{lemma:i} if $Y_{m}^{1},\ldots,Y_{m}^{\delta_{m}}$ is an orthonormal basis of
$\mathcal{H}_{m}\subset L^{2}( \mathcal{M}) $, then%
\[
Z_{x}^{m}( y) =\sum_{\ell=1}^{\delta_{m}}\overline{Y_{m}^{\ell}( x) }Y_{m}^{\ell}(
y) ;
\]

\item the function $Z_{x}^{m}$ is real valued and $Z_{x}^{m}( y) =Z_{y}^{m}( x)$;

\item  \label{lemma:iii} if $g\in G$, then $Z_{gx}^{m}( gy) =Z_{x}^{m}( y)$;

\item \label{lemma:iv} for every $f\in L^2(\mathcal M)$, $Z_o^m \ast f=\Pi_m f$ where $\Pi_m f$ is the orthogonal projection of $f$ on $\mathcal H_m$;

\item \label{lemma:v} $Z_{x}^{m}( x) =\Vert Z_{x}^{m}\Vert_{2}^{2}=\delta_{m}$;

\item $Z_{o}^{m} $ is a zonal function and
\begin{equation}
Z_{o}^{m}( x) =\frac{\delta_{m}}{P_{m}^{a,b}( 1) }P_{m}^{a,b}( \cos( \rho( x,o) ) )
\label{Jacobi}%
\end{equation}
where $P_{m}^{a,b}$ are the Jacobi polynomials;

\item \label{lemma:vii} $\{  \delta_{m}^{-1/2}Z_{o}^{m}\}  _{m=0}^{+\infty}$ is an
orthonormal basis of the subspace of $L^{2}(\mathcal{M})$ of zonal functions;

\item \label{lemma:viii} $\lambda^2_{m}=m(m+a+b+1)$;

\item \label{lemma:ix} and
\begin{align*}
\delta_{m}  &  =  (2m+a+b+1) \frac{\Gamma(  b+1)  }{\Gamma(
a+1)  \Gamma(  a+b+2)  }\frac{\Gamma(  m+a+b+1)
}{\Gamma(  m+b+1)  }\frac{\Gamma(  m+a+1)  }%
{\Gamma(  m+1)  } \asymp m^{2a+1}
\end{align*}
(when $m=0$, the product $(2m+a+b+1)\Gamma(m+a+b+1)$ must be replaced by $\Gamma(a+b+2)$; this way, the case $a=b=-1/2$ can be included). In particular, for $m\ge1$,
\begin{equation}
\label{dimensione}
\delta_m\leq \frac{a+1}{b+1}\left(2m+a+b+1\right)m^{2a}.
\end{equation}
\end{enumerate}
\end{lemma}
\begin{proof} Here we provide a proof only for \eqref{lemma:iv} and the last estimate of \eqref{lemma:ix}, which are not given in the references above.  By \eqref{lemma:i}, \eqref{lemma:iii} and \eqref{eq:251} we have that
\begin{align*}
Z_o^m\ast f(x)&=\int_G Z_o^m (g^{-1}x)f(go)dg=\int_G Z_{go}^m (x)f(go)dg\\
&=\int_G \sum_{\ell=1}^{\delta_m} \overline{Y_{m}^{\ell}( go) }Y_{m}^{\ell}(
x)f(go)dg=\sum_{\ell=1}^{\delta_m} \left(\int_{\mathcal M} f(y)\overline{Y_{m}^{\ell}( y) }d\mu(y)\right) Y_{m}^{\ell}(
x)=\Pi_m f(x),
\end{align*}
as claimed. 

Concerning \eqref{lemma:ix}, inequality \eqref{dimensione} is trivially satisfied for $a=b=-1/2$. In the remaining cases, $a\ge b\ge-1/2$ and $a\ge 0$. Furthermore, either $a+b$ or $a$ are nonnegative integers. If $a+b$ is a nonnegative integer, then
\begin{align*}
\frac{\Gamma(  m+a+b+1)
}{\Gamma(  m+1)  }\frac{\Gamma(  m+a+1)  }
{\Gamma(m+b+1)}&=(m+a+b)\ldots(m+1)(m+a)\ldots(m+b+1)\\
&\le m^{a+b}m^{a-b}(1+a+b)\ldots(2)(1+a)\ldots(1+b+1)\\
&=m^{2a}\frac{\Gamma(a+b+2)\Gamma(a+2)}{\Gamma(b+2)}.
\end{align*}
Similarly, if $a$ is a nonnegative integer, then
\begin{align*}
\frac{\Gamma(  m+a+b+1)
}{\Gamma(  m+b+1)  }\frac{\Gamma(  m+a+1)  }
{\Gamma(m+1)}&=(m+a+b)\ldots(m+b+1)(m+a)\ldots(m+1)\\
&\le m^{a}m^{a}(1+a+b)\ldots(1+b+1)(1+a)\ldots(2)\\
&=m^{2a}\frac{\Gamma(a+b+2)\Gamma(a+2)}{\Gamma(b+2)}.
\end{align*}
Thus, in any case,
\begin{align*}
    \delta_m&\le
    (2m+a+b+1) \frac{\Gamma(  b+1)  }{\Gamma(
a+1)  \Gamma(  a+b+2)  }
m^{2a}\frac{\Gamma(a+b+2)\Gamma(a+2)}{\Gamma(b+2)}\\
&=\frac{a+1}{b+1}\left(2m+a+b+1\right)m^{2a}.
\end{align*}
\end{proof}
Let $f,h$ in $L^2(\mathcal M)$. By Lemma~\ref{lem:zonal_rkhs} \eqref{lemma:vii} we have the following expansion of the zonal function $\widetilde{h}$ given in \eqref{eq:tilde_h}:
\[\widetilde{h}=\sum_{m=0}^{+\infty}b_m Z^m_o,\] 
where
\begin{align*}
b_m&=\frac{1}{\delta_m}\int_{\mathcal M} \widetilde{h}(y)Z_o^m(y)d\mu(y) =\frac{1}{\delta_m}\int_{\mathcal M} \int_K h(ky)Z_o^m(y)dkd\mu(y)\\& =\frac{1}{\delta_m} \int_K \int_{\mathcal M} h(x)Z_{o}^m(k^{-1}x)d\mu(x)dk
=\frac{1}{\delta_m}  \int_{\mathcal M} h(x)Z_{o}^m(x)d\mu(x).
\end{align*}
Also, by \eqref{IntegraleZonale} and \eqref{Jacobi}, we have
\begin{equation}
\label{bm}
b_m=\frac{1}{\delta_m}\int_{\mathcal M} \widetilde{h}(y)Z_o^m(y)d\mu(y)=\int_{0}^{\pi} \widetilde{h}_0(r)\frac{P_{m}^{a,b}( \cos r )}{P_{m}^{a,b}( 1) } A(r)dr . 
\end{equation}
Furthermore, by Lemma~\ref{lem:zonal_rkhs} \eqref{lemma:iv} the convolution of $h$ and $f$ can be written as 
\begin{align*}
h\ast f&=\widetilde{h}\ast f=\sum_{m=0}^{+\infty}b_m Z^m_o \ast f=\sum_{m=0}^{+\infty}b_m\Pi_mf.
\end{align*}
This means that the convolution with a function $h$ can be seen as a multiplier operator. 

\medskip

Let us now set $\mathcal B=\{Y^{\ell}_m\}_{m=0, \ell=1}^{+\infty,\delta_m}$ and $\mathcal E=\{\sqrt{m(m+a+b+1)}\}_{m=0,\ell=1}^{+\infty,\delta_m}$.

\begin{Proposition}\label{operator}
Take $\omega,\gamma\ge 0$. Let $h\in L^2(\mathcal M)$ such that $\widetilde h=\sum_{m=0}^{+\infty}b_m Z^m_o$, with
$$|b_m|\le c\big(1+\lambda_m^2\big)^{-\frac{\gamma}2}$$ for some $c>0$, and $b_0\neq 0$. Let $\sigma=\omega+\gamma$  and consider the linear bounded operator 
given by $Ff=h\ast f$. Then $F$ maps continuously $H^\omega(\mathcal B,\mathcal E)$ into $ H^\sigma(\mathcal B,\mathcal E)$, with norm bounded by $c$.
\end{Proposition}
\begin{proof} By the definition of the norm in $H^\sigma$ (see \eqref{eq:sobolev_norm}) and by Lemma~\ref{lem:zonal_rkhs}, part \eqref{lemma:viii}, we have
\begin{align*}
   \|Ff\|^2_{H^\sigma}&= \|\widetilde h\ast f\|^2_{H^\sigma}=\left\|\sum_{m=0}^{+\infty}b_m\Pi_mf\right\|^2_{H^\sigma}=\sum_{m=0}^{+\infty}b_m^2\|\Pi_mf\|_2^2(1+\lambda_m^2)^\sigma\\
   &\le c^2\sum_{m=0}^{+\infty}\|\Pi_mf\|_2^2(1+\lambda_m^2)^{\sigma-\gamma}=c^2\|f\|^2_{H^\omega}.\qedhere
\end{align*}
\end{proof}

\subsection{Sampling convolutions}\label{sub:sampling_convolutions}

The focus of the following result is the reconstruction of $Ff^\dagger$ from some noisy sampled measurements, where $F$ is a convolution operator.

\begin{Theorem}\label{finale}
Let $\mathcal M$ be a $d$-dimensional compact two-point homogeneous space as described above,  $d\ge2$. Assume that the nodes $\{x_{m,j}\}_{j=1}^{L_m}$, the weights $\{\tau_{m,j}\}_{j=1}^{L_m}$ and the sequence $r_m=\lambda_m$ are a Marcinkiewicz-Zygmund family for $\mathcal B$ and $\mathcal E$ with constants $1-\varepsilon$ and $1+\varepsilon$, with $\varepsilon>0$. 

Let $h\in L^2(\mathcal M)$ be such that $\widetilde h(x)=\sum_{m=0}^{+\infty}b_m Z^m_o(x)$, with $|b_m|\le c(1+\lambda_m^2)^{-\gamma/2}$ for some $c>0$, $\gamma\ge0$, and $b_m\neq 0$ for infinitely many values of $m$, including $m=0$. Let $\omega\ge0$, $\sigma=\omega+\gamma$  and consider the linear bounded operator $F\colon H^\omega(\mathcal B,\mathcal E)\to H^\sigma(\mathcal B,\mathcal E)$
given by $Ff=h\ast f$. Let $\zeta\ge0$, assume $\sigma-\zeta>d/2$ and take $f^\dagger\in H^{\omega}(\mathcal B,\mathcal E)$. Let $\{y_{m,j}\}$ be a
family of noisy outputs corresponding to $\{Ff^\dagger(x_{m,j})\}$,
\[
 |y_{m,j}- Ff^\dagger(x_{m,j})|\leq\beta \quad j=1,\ldots,L_m
\]
where $\beta \ge 0$ is the noise level. For  every $m\in\mathbb N$ let
\[
p_m^\beta\in\operatornamewithlimits{argmin}_{p\in \mathcal Q_{\lambda_m}(\mathcal B,\mathcal E)}\sum_{j=1}^{L_{m}}|y_{m,j}-Fp(x_{m,j})|^2\tau_{m,j}.
\]
Then
\begin{align*}
\|Fp_m^\beta-Ff^\dagger\|_{H^{\zeta}(\mathcal B,\mathcal E)}&\le  \sqrt{1+\kappa}\|Ff^\dagger\|_{H^\sigma(\mathcal B, \mathcal E)}\phi_{\sigma-\zeta}(\lambda_m,\mathcal B,\mathcal E)+\sqrt{\kappa}\beta(1+\lambda_m^2)^{\frac\zeta 2}\\
&\le 
\sqrt{\frac{d(1+\kappa)}{d_0(\sigma-\zeta-d/2)}}\frac{c}{(1+\lambda_m^2)^{(\sigma-\zeta-d/2)/2}}\|f^\dagger\|_{H^\omega(\mathcal B, \mathcal E)}+\sqrt{\kappa}\beta(1+\lambda_m^2)^{\frac\zeta 2}.
\end{align*}
Recall that $\kappa=\frac{1+\varepsilon}{1-\varepsilon}$, and
$d/d_0=1$ for the sphere, while $d/d_0=n$ for the projective space $P^n(\mathbb K)$.
\end{Theorem}

The existence of a Marcinkiewicz-Zygmund family, as in the hypotheses of Theorem~\ref{finale}, is guaranteed by Theorems~\ref{MZ polys} and \ref{partition}.

\begin{proof} 
We want to apply Corollary \ref{noisebis}. Assumption~\ref{spazio Y} is satisfied by hypothesis. In particular, notice that the space $\mathcal Q_{\lambda_m}(\mathcal B,\mathcal E)$ coincides with
the space $G_m:=\bigoplus_{s=0}^{m}
\mathcal{H}_{s}$
and Assumption \ref{spazio X} follows.
Finally, since $F(G_m)\subseteq G_m$, $1\in F(G_0)$ (since $G_0=\mathcal{H}_0=\operatorname{span}\{1\}$ and $b_0\neq 0$), and $\dim \operatorname{Im} F=+\infty$ (since $b_m\neq0$ for infinitely many values of $m$), then also Assumption~\ref{operatore F}  holds. 

 We may therefore apply Corollary \ref{noisebis}, with $\mathcal S=\{Y_m^\ell: m=0,\ldots,+\infty \text{  s.t. } b_m\neq 0, \ell=1,\ldots,\delta_m\}\subseteq \mathcal B$  and $\mathcal L=\{\sqrt{m(m+a+b+1)}: m=0,\ldots,+\infty \text{  s.t. } b_m\neq 0, \  \ell=1,\ldots,\delta_m\}\subseteq\mathcal E$ and we obtain
\[
\|Ff^\dagger-Fp_m^\beta\|_{H^\zeta(\mathcal S,\mathcal L)}\le \sqrt{1+\kappa}\|Ff^\dagger\|_{H^\sigma(\mathcal S, \mathcal L)}\phi_{\sigma-\zeta}(\lambda_{m}, \mathcal S, \mathcal L)+\sqrt{\kappa}\beta(1+\lambda_m^2)^{\frac\zeta2}.
\]
Notice that $\|Ff^\dagger\|_{H^\sigma(\mathcal S,\mathcal L)}=\|Ff^\dagger\|_{H^\sigma(\mathcal B,\mathcal E)}<+\infty$ by Proposition~\ref{operator} and that, since $\sigma>d/2+\zeta\ge d/2$, then $Ff^\dagger\in H^{\sigma}(\mathcal B,\mathcal E)$ is continuous. Also,
\begin{align*}
\phi_{\sigma-\zeta}(\lambda_{m},\mathcal S,\mathcal L)&=\sup_{x\in\mathcal M}\left(\sum_{n>m,\ b_n\neq 0}\sum_{\ell=1}^{\delta_n}|Y_n^\ell(x)|^2(1+\lambda_n^2)^{-(\sigma-\zeta)}\right)^{\frac12}\\
&\leq  \sup_{x\in\mathcal M}\left(\sum_{n>m}\sum_{\ell=1}^{\delta_n}|Y_n^\ell(x)|^2(1+\lambda_n^2)^{-(\sigma-\zeta)}\right)^{\frac12}=\phi_{\sigma-\zeta}(\lambda_m,\mathcal B,\mathcal E).
\end{align*}%
Since $ Z_x^n(x)=\delta_n$ by Lemma~\ref{lem:zonal_rkhs} \eqref{lemma:v}, and since by \eqref{dimensione}
\[
\delta_n\le \frac{a+1}{b+1}\left(2n+a+b+1\right)(1+\lambda_n^2)^{a},
\]
by  Lemma~\ref{lem:zonal_rkhs} \eqref{lemma:i} and \eqref{eq:abdd0}  it follows that
\begin{align*}
\phi_{\sigma-\zeta}(\lambda_m,\mathcal B,\mathcal E)^2&=\sup_{x}
\sum_{n>m}\sum_{\ell=1}^{\delta_n}|Y_n^\ell(x)|^2
(1+\lambda_n^2)^{-(\sigma-\zeta)}
\\&=\sup_x\sum_{n>m}Z_x^n(x)(1+\lambda_n^2)^{-(\sigma-\zeta)}\\
&\leq\frac{a+1}{b+1} \sum_{n>m} (2n+a+b+1)(1+\lambda_n^2)^{a+\zeta-\sigma}\\ 
&\leq \frac{a+1}{b+1}\int_{m}^{+\infty} (2x+a+b+1)(1+x(x+a+b+1))^{a+\zeta-\sigma} \,dx\\ 
&=\frac{a+1}{b+1}\int_{1+m(m+a+b+1)}^{+\infty} y^{a+\zeta-\sigma} dy\\ 
&=\frac{a+1}{b+1}\frac{1}{\sigma-\zeta-a-1}\frac{1}{(1+m(m+a+b+1))^{\sigma-\zeta-a-1}}\\
&=\frac{d}{d_0}\frac{1}{\sigma-\zeta- d/2}\frac{1}{(1+\lambda_m^2)^{\sigma-\zeta-\frac{d}2}}.
\end{align*}
This concludes the proof. 
\end{proof}

\begin{remark}
It is clear from the above proof that the term $\phi_{\sigma-\zeta}(\lambda_m,\mathcal B,\mathcal E)$ in the thesis of the theorem can be replaced by $\phi_{\sigma-\zeta}(\lambda_m,\mathcal S,\mathcal L)$. Of course, when several terms of the sequence $b_m$ vanish, then $\phi_{\sigma-\zeta}(\lambda_m,\mathcal S,\mathcal L)$ may be asymptotically smaller than $\phi_{\sigma-\zeta}(\lambda_m,\mathcal B,\mathcal E)$.
\end{remark}

In the following result, we investigate the reconstruction of $f^\dagger$.

\begin{Corollary}\label{inverso}
Under the hypotheses of Theorem~\ref{finale}, assume furthermore that $\zeta\ge\gamma$ and that, for some constant $c_0>0$,
\[
|b_m|\ge c_0(1+\lambda_m^2)^{-\zeta/2}, \text{ for all }m\ge0.
\]
Then
\begin{align*}
\|f^\dagger-p_m^\beta\|_{2}&\le c_0^{-1} \|Ff^\dagger-Fp_m^\beta\|_{H^\zeta(\mathcal B, \mathcal E)}\\
&\le  c_0^{-1}\sqrt{1+\kappa}\|Ff^\dagger\|_{H^\sigma(\mathcal B, \mathcal E)}\phi_{\sigma-\zeta}(\lambda_m,\mathcal B,\mathcal E)+\sqrt{\kappa}\beta(1+\lambda_m^2)^{\frac\zeta 2}\\
&\le 
\frac{c}{c_0}\sqrt{\frac{d(1+\kappa)}{d_0(\sigma-\zeta-d/2)}}\frac{1}{(1+\lambda_m^2)^{(\sigma-\zeta-d/2)/2}}\|f^\dagger\|_{H^\omega(\mathcal B, \mathcal E)}+\sqrt{\kappa}\beta(1+\lambda_m^2)^{\frac\zeta 2}.
\end{align*}
\end{Corollary}
\begin{proof}
Note that, since $b_m\neq 0$ for every $m$, we have that $F$ is injective, with inverse given by
\[
F^{-1} f = \sum_{m} b_m^{-1} \Pi_m f.
\]
Thus
\[
\|F^{-1} f^\dagger\|_2^2 =\sum_{m} |b_m|^{-2} \|\Pi_m f^\dagger\|_2^2 
\le c_0^{-2} \sum_{m} (1+\lambda_m^2)^{\zeta} \|\Pi_m f^\dagger\|_2^2
=  c_0^{-2} \|f^\dagger\|_{H^\zeta}^2.
\]
Therefore,
\[
\|f^\dagger-p_m^\beta\|_2 \le c_0^{-1} \|Ff^\dagger-Fp_m^\beta\|_{H^\zeta}.
\]
Applying this to the estimates in Theorem~\ref{finale}, the thesis follows.
\end{proof}

The following proposition provides an alternative description of
Sobolev spaces and an explicit formula for computing the coefficients
$b_m$. 
\begin{Proposition}
\label{bim}
Let $K$ be a nonnegative integer. Then 
\begin{enumerate}[(i)]
\item \label{prop:i}  $H^{2K}(\mathcal B,\mathcal E)=\left\{f\in L^2(\mathcal M): \ (-\Delta)^K f\in  L^2(\mathcal M)\right\}$;
\item if $h=\sum_{m=0}^{+\infty} b_mZ^m_o \in H^{2K}(\mathcal B,\mathcal E)$, then 
\[
b_0=\int_{\mathcal M} h(x)d\mu(x)
\]
and for $m\ge1$ 
\[b_m=\frac{1}{\delta_m(m(m+a+b+1))^{K}}\int_{\mathcal M} (-\Delta)^K h(x)Z^m_o(x)d\mu(x);\]
\item for $m\ge 1$, \[|b_m|\leq \frac{\|(-\Delta)^K h\|_2}{\delta_m^{\frac12}(m(m+a+b+1))^{K}}\leq
\frac{2^{K}\|(-\Delta)^K h\|_2}{\delta_m^{\frac12}(1+m(m+a+b+1))^{K}}.
\]
\end{enumerate}
\end{Proposition}
\begin{proof}
The equivalence in \eqref{prop:i} is well known, see, for example, \cite{MR1688256}. The rest follows from the definition of $b_m$ as
\[
b_m=\frac 1{\delta_m}\int_{\mathcal M} h(x) Z^m_o(x) d\mu(x),
\]
from the self-adjointness of $\Delta$ and the fact that $(-\Delta)^K Z^m_o = (m(m+a+b+1))^K Z^m_o$.
\end{proof}
\begin{remark}\label{bmsfera}
Notice that in the case of the sphere $S^2$, if $h\in H^{2K}(\mathcal B,\mathcal E)$ for some integer $K$, then by Proposition~\ref{bim},
for $m\ge 1$
\[
|b_m|\le\frac{2^{K}\|(-\Delta)^K h\|_2}{(2m+1)^{\frac12}(1+m(m+1))^{K}}\leq \frac{2^{K}\|(-\Delta)^K h\|_2}{(1+m(m+1))^{\frac{4K+1}4}}.
\]
\end{remark}

\section*{Acknowledgments}
Co-funded by the European Union (ERC, SAMPDE, 101041040). Views and opinions expressed are however those of the authors only and do not necessarily reflect those of the European Union or the European Research Council. Neither the European Union nor the granting authority can be held responsible for them. The authors are members of the “Gruppo Nazionale per l’Analisi Matematica, la Probabilità e le loro Applicazioni”, of the “Istituto Nazionale di Alta Matematica”.
The research was supported in part by the MIUR Excellence Department Project awarded to Dipartimento di Matematica, Università di Genova, CUP D33C23001110001. Finanziato dall’Unione europea-Next Generation EU, Missione 4 Componente 1 CUP D53D23005770006. B. Gariboldi and G. Gigante are supported by the PRIN 2022 project "TIGRECO - TIme-varying signals on Graphs: REal and COmplex methods" funded by the MUR (Ministero dell'Università e della Ricerca), Grant 20227TRY8H, CUP F53D23002630001. The research by G.S.\ Alberti and E.\ De Vito has been partially supported by the PNRR project “Harmonic Analysis and Optimization in Infinite-Dimensional Statistical - Future Artificial Intelligence  Fair – Spoke 10” (CUP J33C24000410007). The research by E.\ De Vito has been partially supported by the MIUR grant PRIN 202244A7YL.

\bibliographystyle{abbrv}
\bibliography{refs}

\begin{thebibliography}{10}
\bibitem{Whittaker1915}
E.~T. Whittaker.
\newblock On the functions which are represented by the expansions of the
  interpolation theory.
\newblock {\em Proceedings of the Royal Society of Edinburgh}, 35:181--194,
  1915.
\end{thebibliography}

\end{document}